\documentclass[10pt]{article}

\usepackage{amssymb,amsmath,amsfonts,amsthm,enumerate,color,mathrsfs,hyperref,amssymb}
\usepackage[toc,page,title,titletoc,header]{appendix}
\usepackage{graphicx}
\usepackage{indentfirst}
\usepackage{multicol}
\usepackage{booktabs}
\usepackage{algorithm}
\usepackage{algorithmic}
\usepackage{empheq}
\usepackage{subfigure}

\newcommand{\mb}[1]{{\color{black}{#1}}}

\numberwithin{equation}{section}

\theoremstyle{plain}
\newtheorem{thm}{Theorem}[section]
\newtheorem{lem}[thm]{Lemma}

\newtheorem{prop}[thm]{Proposition}
\theoremstyle{remark}

\newtheorem{re}[thm]{Remark}

\theoremstyle{definition}

\newcommand{\weak}{\stackrel{w}{\rightharpoonup}}
\newcommand{\norm}[1]{\lVert #1 \rVert}
\newcommand{\T}[1]{\mathscr{T}_{#1}}
\newcommand{\Rmnum}[1]{\uppercase\expandafter{\romannumeral#1}}

\def \sss{\scriptscriptstyle}

\def\Om{\Omega}

\def\R{{\mathbb R}}

\def\u{\textit{\textbf u}}

\def\x{\textit{\textbf x}}
\def\a{\textit{\textbf a}}
\def\e{\textit{\textbf e}}
\def\b{\textit{\textbf b}}
\def\U{\textit{\textbf U}}
\def\D{\textit{\textbf D}}

\def\y{\textit{\textbf y}} \def\V{\textit{\textbf V}}
\def\g{\textit{\textbf g}}
  \def\f{\textit{\textbf{f}}} \def\n{\textit{\textbf n}}
\def\p{\textit{\textbf p}} \def\v{\textit{\textbf{v}}} \def\w{\textit{\textbf w}}
\def\z{\textit{\textbf z}}

\def \h{\hat}

\def \phii{ \boldsymbol{\phi}}
\def \varphii{ \boldsymbol{\varphi}}
\def\yk{\y_k^*}
\def\pk{\p_k^*}
\def\uk{\u_k^*}
\def\llambda {\boldsymbol{\lambda} }
\def \curl{{\bf curl}}
\def \div{{\rm div}}
\def \0{{\bf 0}}

\def \osc{{\rm osc}}

\title{An adaptive edge element method and its convergence \\ for an 
electromagnetic constrained optimal control problem}

\begin{document}
\author{
Bowen Li\footnote{Department of Mathematics, The Chinese University of Hong Kong, Shatin, N.T., Hong Kong. (bwli@math.cuhk.edu.hk)}
\and Jun Zou\footnote{Department of Mathematics, The Chinese University of Hong Kong, Shatin, N.T., Hong Kong.
The work of this author was
substantially supported by Hong Kong RGC General Research Fund (projects  14306718 and 14306719).
(zou@math.cuhk.edu.hk).}
}

\date{}
\maketitle

\begin{abstract}
In this work, an adaptive edge element method is developed for an $\textit{\textbf{H}}(\curl)$-elliptic constrained optimal control problem. 
We use the lowest-order N\'{e}d\'{e}lec's edge elements of  first family and the piecewise (element-wise) constant functions to approximate the state and the control, respectively, and propose a new   
adaptive algorithm with error estimators involving both residual-type error estimators and lower-order data oscillations. 
By using a local regular decomposition for $\textit{\textbf{H}}(\curl)$-functions and the standard bubble function techniques, we  derive the a posteriori error estimates for the proposed error estimators.
Then we exploit the convergence properties of 
the orthogonal $L^2$-projections and the mesh-size functions to demonstrate that the sequences of the discrete states and
controls generated by the adaptive algorithm converge strongly to the exact solutions of the state and control in the energy-norm and $L^2$-norm, respectively, by first achieving the strong convergence towards the solution to a limiting control problem. Three-dimensional  numerical experiments are also presented to confirm our theoretical results and the quasi-optimality of the adaptive
edge element method.
\end{abstract}

\noindent { \footnotesize {\bf Mathematics Subject Classification
(MSC2000)}: 65K10, 65N12, 65N15, 65N30, 49J20}


\section{Introduction} \label{introduction}
Many electromagnetic simulation problems involve the following $H({\bf curl})$-elliptic equation
\cite{bossavit1998computational} \cite{monk2003finite}:
\begin{align}\label{eddy}
{\bf curl} \left(\mu^{-1} {\bf curl} \y\right) + \sigma \y = \f   \quad &  \text{in} \ \Om \,,
\end{align}
where $\sigma$ and $\mu$ are the electric permittivity and the magnetic permeability, respectively. 
This problem is also encountered when the implicit time-stepping scheme is used for solving the full time-dependent Maxwell system \cite{hiptmair1998multigrid}. In this work, we are interested in a relevant optimal control problem, namely
to find a specially designed external current source profile so that the resulting electromagnetic field achieves an optimal target design. This optimal control problem has direct applications in many areas, such as the induction heating, the magnetic levitation and the electromagnetic material designs. We refer the readers to \cite{yousept2013optimal}\cite{bommer2016optimal}\cite{yousept2017optimal}\cite{fang2020optimal} for some recent results on the theory and applications of  electromagnetic optimal control problems.

The main purpose of this work is to design and analyse an adaptive edge element method for a class of optimal control problems with the applied current density $\u$ 
being the control variable and the magnetic potential $\y$ being the state variable. Let $\y^d$ and $\u^d$ be our target magnetic field and control, and $\U^{ad}$ denote the unbounded closed convex admissible set:
\begin{equation} \label{eq:addset}
\left\{\u \in \textit{\textbf{L}}^2\left(\Om\right)\,;\ \u \ge \boldsymbol{\psi}  \ \text{a.e.} \ \text{in} \ \Om\right\}\,,
\end{equation}
where $\boldsymbol{\psi}$ is some obstacle function with a certain regularity.
Then we can write the control problem as a constrained minimization problem:
find a pair $(\y^*,\u^*) \in \textit{\textbf{H}}_0(\curl,\Omega) \times\U^{ad}$ to minimize the quadratic objective functional:
\begin{equation}\label{introop1}
J(\y,\u) := \frac{1}{2}\norm{{\bf curl} \y-\y^d}_{0,\Om}^2+\frac{\alpha}{2}\norm{\u-\u^d}_{0,\Om}^2\,,
\end{equation}
subject to the $H({\bf curl})$-elliptic equation:
\begin{equation} \label{introop2}
    {\bf curl} \left(\mu^{-1} {\bf curl} \y\right) + \sigma \y = \f + \u  \quad  \text{in} \ \Om \,.
\end{equation}
By considering a simple transformation $\widetilde{\u} = \u - \boldsymbol{\psi}$, we may set $\boldsymbol{\psi}$ in \eqref{eq:addset} to zero in our subsequent analysis for the sake of simplicity, that is, the admissible set shall take the form:
\begin{equation*}
    \U^{ad}: = \left\{\u \in \textit{\textbf{L}}^2\left(\Om\right)\,;\ \u \ge \0 \ \text{a.e.} \ \text{in} \ \Om\right\}.
\end{equation*}
In this work, we are interested in the numerical study of the practically important situation where 
the local singularities may be expected for the solution to the optimal control problem \eqref{introop1}-\eqref{introop2}, due to the possible irregular geometry of the domain $\Om$, the (possibly large) jumps across the interface 
between two different physical media and the non-smooth source terms (cf.\,\cite{costabel2000singularities}\cite{costabel1999singularities}). 
In these cases, one typically needs to solve the problem on a very fine mesh to fully resolve the singularities, while it may be computationally expensive 
on account of the exponential growth of the number of degrees of freedom (DoFs).
To balance the computational cost and the  numerical accuracy, it has proved to be more promising to refine the mesh adaptively when the numerical accuracy is insufficient.


The theory of adaptive finite element methods (AFEMs),  due to the pioneer work of \mb{Babu$\check{{\rm s}}$ka} and Rheinboldt \cite{babuvvska1978error} in 1978 for elliptic boundary value problems, has become very popular and well developed
in the past four decades; see \cite{ainsworth2011posteriori}\cite{nochetto2009theory} 
and the references therein for an overview. The  
relevant research in the case of Maxwell's equations, which dated back to Beck et al.\,\cite{beck2000residual},  has also reached a mature level nowadays (cf.\,\cite{zhong_convergence_2012}\cite{zhong_convergence_2010}\cite{schoberl2008posteriori}\cite{hoppe2009convergence}). However, the theory of AFEMs for  PDE-constrained optimization problems is still a work in progress. 
Recently, Gong and Yan \cite{gong2017adaptive} proved the convergence and quasi-optimality of AFEMs for an elliptic optimal control problem with the help of the variational discretization and the duality argument, and in this setting there is no need to mark the data oscillations as in \cite{hintermuller2008posteriori}\cite{gaevskaya2007convergence}. In \cite{leng2017convergence}\cite{leng2018convergence}, the authors considered the elliptic optimal control problems with integral control constraints and gave a rigorous proof of the convergence and the  quasi-optimality.  
People may find more recent advances on AFEMs for the control problems in \cite{becker2011quasi}\cite{kohls2014posteriori}\cite{allendes2017adaptive}\cite{gong2018adaptive}.
 
The first contribution towards the adaptive methods for optimal control problems of Maxwell's equations went back to Hoppe and Yousept \cite{hoppe2015adaptive}. An AFEM was designed for the same model problem 
(\ref{introop1})-(\ref{introop2}), 
and the a posteriori error estimation was also presented, including the reliability and efficiency, 
with N\'{e}d\'{e}lec's  edge element approximations for both the control and the state. But the convergence of the adaptive method was not established in \cite{hoppe2015adaptive}.
Later, Pauly and Yousept \cite{pauly2017posteriori} considered an optimal control problem of first-order magnetostatic equations and presented the a posteriori error analysis. Meanwhile, Xu and Zou \cite{xu2017convergent} considered the adaptive approximation for an optimal control problem associated with  an electromagnetic saddle-point model and proved the convergence of discrete solutions.



The current work is a continuation and further improvement of the work \cite{hoppe2015adaptive}, where a piecewise linear $\textit{\textbf{H}}(\curl)$-conforming edge element was used to approximate 
the optimal control variable $\u$.
One crucial motivation of our work is based on the observation that the optimal control $\u^*$ is generally of 
low regularity, at most $\textit{\textbf{H}}^1$ (still under some reasonable conditions) (cf.\,Proposition \ref{prop:reg}).
Therefore we can not expect a better $L^2$- approximation accuracy of the control 
using any higher-order approximation than the piecewise constant one (i.e., the simplest, and also cheapest, 
approximation).
Because of this, we shall use the piecewise constant functions to approximate the control variable $\u$, and propose a new adaptive algorithm with error indicators incorporating the residual-type error estimators and the lower-order data oscillations.  As pointed out in \cite{kohls2014convergence}\cite{Kohls-Kreuzer-Rosch-Siebert2018}, the precise structure of the admissible set and the way we discretize the distributed control problem are the critical ingredients when designing an adaptive
algorithm, which often influence the analysis of the resulting AFEM essentially. It turns out that in our case here, due to the pointwise unilateral constraints in the admissible set $\U^{ad}$, 
the inconsistency between the discrete spaces of the state and control as well as the low regularity of the given data, the data oscillation plays an important role in the performance of the algorithm and has to be considered in the error estimates (as well as the marking strategy) to guarantee the convergence of the algorithm. We can readily notice that such 
a change in the approximation of the control can significantly reduce the number of DoFs and make our algorithm much easier to implement than the algorithm in \cite{hoppe2015adaptive}, while still preserving 
the same numerical accuracy; see Section \ref{experiment} for the detailed discussions and the implementation 
issue on the algorithm. 
To theoretically justify the effectiveness of our adaptive algorithm, we use a Cl\'{e}ment-type quasi-interpolation operator established by Sch\"{o}berl in \cite{schoberl2008posteriori} and the bubble functions \cite{ainsworth2011posteriori} to derive the a posteriori error estimates (reliability and efficiency estimates) of the new error estimators. It is worth mentioning that there is a technical issue in \cite{hoppe2015adaptive} when deriving the optimality system for the discrete optimization problem, that is, the first-order optimality condition (2.8d) there cannot imply its complementarity condition (2.10) 
as the edge element discretization was adopted 
for the control (hence part of the estimates provided in \cite{hoppe2015adaptive} needs some necessary modifications). This is another motivation of the current work. Instead, in our a posteriori analysis, we shall directly cope with the original forms of the  first-order optimality conditions by using the  contraction properties of  $L^2$-projections without the help of the complementarity problems (although they hold in our setting), so that our estimates can be established in 
a rigorous manner. Apart from the aforementioned contributions towards the algorithmic aspects,
%
another main contribution of this work is the strong convergence of the finite element solutions generated by 
the proposed adaptive method.  This work appears to be the first one on the convergence analysis of AFEMs applied to the electromagnetic optimal control problem involving a variational inequality structure. In the finite element analysis provided in \cite{yousept2013optimal} for a quasilinear $\textit{\textbf{H}}(\curl)$-elliptic optimal control problem, the discrete compactness of the N\'{e}d\'{e}lec edge elements \cite{kikuchi1989discrete} is the main tool in establishing strong convergence. However, the discrete compactness can only be applied to the discrete divergence-free sequence, and it is still unknown whether it is true for adaptively generated meshes. During the course of our convergence analysis, we shall borrow some techniques from the nonlinear optimization to avoid the need of discrete compactness 
and exploit some strategies with limiting spaces, which was initially adopted 
in 
\cite{bubuvska1984feedback} and then developed systemically in 
\cite{morin2008basic}. To overcome the difficulties arising from the structure of the constraint set and its discretization, we establish a convergence property of $L^2$-projections associated with the meshes (cf.\,Proposition \ref{average}), which connects the convergence behaviors of the adaptive meshes and the discrete spaces. 
It further allows us to define a limiting optimization problem and 
characterize the limit point of the discrete controls and states (cf.\,Theorems \ref{thm:limweak} and \ref{limiting}). Hereafter, we show that the maximal error estimators and the residuals corresponding to the sequence of  adaptive states and adjoint states vanish (cf.\,Lemmas \ref{lem:vaniestimator} and \ref{btwo}), where we use a generalized convergence result of the mesh-size functions (cf.\,\eqref{uniformstrong}), instead of introducing the buffer layer (cf.\,\cite{xu2017convergent}) between the meshes at different levels, so that our proof can be significantly simplified. By means of
these auxiliary results and again the properties of $L^2$-projections, we are able to prove that the limit point of the discrete triplets $\{(\yk,\pk,\uk)\}$ solves the variational inequality associated with the control problem (cf.\,Theorem \ref{thm:mainresult}). Hence, the strong convergence of Algorithm \ref{alg:Framwork}, i.e., Theorem \ref{thm:converge}, follows immediately.



This work is organized as follows. In Section \ref{problem}, we briefly review the electromagnetic optimal control problem and 
consider its finite element approximation.
Then we propose the a posteriori error estimator and design an adaptive algorithm. We show that the error estimator is both  reliable and efficient in Section \ref{algorithm}. 
After that, we address the issue of the convergence of the adaptive solutions
in Section \ref{converge}. Finally, some numerical results are presented in Section \ref{experiment} to illustrate the theoretical results and indicate the effectiveness and robustness of the adaptive approach versus the uniform mesh refinement. We end this section with a shorthand notation: $x\lesssim y$ for $x \le C y$ for some generic constant $C$ that is independent of the mesh size, but may depend on other quantities, e.g., $\sigma$, $\mu^{-1}$, $f$ and the shape regularity of
the initial triangulation $\T{0}$.  If $x\leq C y$ and $x\geq C y$ hold simultaneously, we denote it simply by $x \approx y$.


\section{The optimal control problem and its discretization}\label{problem}
We start by introducing the standard notation for Sobolev spaces 
and the involved physical parameters (cf.\,\cite{adams2003sobolev}\cite{monk2003finite}). Let  $\Omega \in \R^3$ be a bounded polyhedral domain with polyhedral Lipschitz subdomains $\Omega_i, 1\leq i \leq m$, such that
\begin{equation*}
    \Om_i \cap \Om_j = \varnothing \ \text{for} \ i\neq j \quad \text{and} \quad \bar{\Om} = \bigcup_{i = 1}^m \bar{\Om}_j\,.
\end{equation*}
\mb{For any open bounded subset $G$ of $\Omega$ with a Lipschitz continuous boundary $\partial G$ 
and $\n$ being its unit outward normal vector.} For any real $s > 0$, we define the Hilbert space $H^s(G)$ (resp.\,$\textit{\textbf{H}}^s(G):= H^s(G,\R^3)$) for Sobolev scalar functions (resp.\,vector fields) of order $s$ with an inner product $(\cdot ,\cdot)_{s,G}$ and a norm $\norm{\cdot}_{s,G}$. 
We also introduce the spaces: 
$$\textit{\textbf{H}}(\curl,G) = \{\v \in \textit{\textbf{L}}^2(G)\,;\ \curl\v \in \textit{\textbf{L}}^2(G) \} \quad \text{and} \quad  \textit{\textbf{H}}(\div,G) = \{\v \in \textit{\textbf{L}}^2(G)\,;\ \div\v \in L^2(G)\},$$ equipped with the norm
$\norm{\v}_{\curl,G} = (\norm{\v}_{0.G}^2 + \norm{\curl\v}_{0,G}^2)^{1/2}$ and $\norm{\v}_{\div,G} = (\norm{\v}_{0.G}^2 + \norm{\div\v}_{0,G}^2)^{1/2}$, respectively. Here and in what follows, a bold typeface is used to indicate a vector-valued function. The tangential trace mapping $\gamma_t(\u) := \n \times \u$ and the normal trace mapping $\gamma_n(\u) := \u \cdot \n$ are well-defined on \mb{$\textit{\textbf{H}}(\curl,G)$ and $\textit{\textbf{H}}(\div,G)$}, respectively. 
Then the zero tangential trace space can be introduced by 
$$ \textit{\textbf{H}}_0(\curl,G) = \{\v \in \textit{\textbf{H}}(\curl,G) \,;\  \gamma_t(\v)=0 \ \text{on} \ \partial G \}.$$ 
We will also use the fractional order $\curl$-space 
\begin{align} \label{defhsspace}
    \textit{\textbf{H}}^s(\curl,\Omega) =\{\u \in \textit{\textbf{H}}^s(\Omega)\,; \ \curl \u \in \textit{\textbf{H}}^s(\Omega)\}\,,\quad s > 0\,,
\end{align} 
equipped with the norm $\norm{\cdot}_{H^s(\curl,\Omega)} := \norm{\cdot}_{s,\Omega} + \norm{\curl \cdot}_{s,\Omega}$.
In what follows, we write $\V$ for the most frequently used Sobolev space $\textit{\textbf{H}}_0(\curl,\Om)$.

We are now well-prepared for the mathematical formulation of the control problem.
In this work, we focus on the following constrained minimization problem:
\begin{equation} \label{target:ctn}
\min_{\u \in \U^{ad}} J(\u) := \frac{1}{2}\norm{\curl \y(\u)-\y^d}_{0,\Om}^2+\frac{\alpha}{2}\norm{\u-\u^d}^2_{0,\Om}\,,
\end{equation}
where the state \mb{$\y(\u) \in \V$ is the solution to the} variational system: 
\begin{equation} \label{constraint}
 (\mu^{-1} \curl \y ,\curl \phii)_{0,\Omega}+ (\sigma \y , \phii)_{0,\Omega} = (\f + \u ,\phii)_{0,\Omega} \quad \forall \ \phii \in \V\,.
\end{equation}
We assume that the given source term $\f \in \textit{\textbf{L}}^2(\Omega)$  satisfies 
\begin{equation} \label{aux_assp}
\f\,|_{\Omega_i} \in \textit{\textbf{H}}(\div,\Om_i)\quad  \text{and}\quad \mb{\gamma_n(\f) \in \textit{\textbf{L}}^2(\partial \Omega_i)}\,, \quad  1 \le i \le m\,,
\end{equation} 
while the target magnetic field $\y^d$ and the target control $\u^d$ satisfy
\begin{align} \label{asspdata}
    \y^d \in \textit{\textbf{H}}_0(\curl,\Om) \quad \text{and} \quad \u^d \in \textit{\textbf{L}}^2\left(\Om\right). 
\end{align}
\mb{We remark that although $\f$ is not in $\textit{\textbf{H}}(\div)$ globally in $\Omega$, namely 
$\f \notin \textit{\textbf{H}}(\div,\Omega)$, $ \div \f$ is well-defined on each $\Omega_i$ (hence almost everywhere on $\Omega$) by the assumption \eqref{aux_assp}. In what follows, we write $\div \f\,$  for $\sum_{i = 1}^m \div \f\, \chi_{\Omega_i}$, by abuse of notation, where $\chi_{\Omega_i}$ is the characteristic function of $\Omega_i$.
It is also clear that the weak divergence of $\f\,$ is globally defined on $\Omega$ if and only if the jump of the normal trace $[\gamma_n(\f)]_\Gamma$ across the interfaces $\Gamma: = \cup_{i = 1}^m \partial \Omega_i \backslash \partial \Omega$ vanishes.} The physical parameters $\mu$ and $\sigma$ are supposed to be polynomials on each subdomain $\Om_i$ (hence piecewise polynomials on $\Om$) and satisfy $\mu(x)\geq\mu_0>0$ and $\sigma(x)\geq \sigma_0>0$.
The real $\alpha > 0$ is a stabilization parameter. We define the symmetric bilinear form $B(\cdot,\cdot)$ associated with the state equation \eqref{constraint}:
\begin{equation*}
B(\u,\v) :=  (\mu^{-1} {\bf curl} \u , {\bf curl} \v)_{0,\Omega}+ (\sigma \u , \v)_{0,\Omega}\,.
\end{equation*}
It readily follows from the assumptions of the physical parameters that $B$ is continuous and coercive, thus defining an inner product on $\textit{\textbf{H}}(\curl,\Omega)$ whose induced energy norm $\norm{\cdot}_B:= \sqrt{B(\cdot,\cdot)}$ is equivalent to $\norm{\cdot}_{\curl,\Om}$.

By the direct method in the calculus of variations (see, for instance, Theorems 2.14--2.16 in \cite{troltzsch2010optimal}) and noting that 
the cost functional $J(\u)$ is weakly lower semicontinuous and strictly convex, we have that there always exists a unique minimizer $\u^* \in \U^{ad}$ to the optimization problem. Then the first-order optimality condition yields
\begin{equation}\label{variin}
(\curl\y(\u^*)-\y^d,\curl\y(\u)-\curl\y(\u^*))_{0,\Omega}+\alpha(\u^*-\u^d,\u-\u^*)_{0,\Omega}\geq 0 \quad \forall \ \u \in \U^{ad}\,,
\end{equation}
since $J(\u)$ is Fr\'{e}chet differentiable. 
In what follows, we shall write $\y^*$ for $\y(\u^*)$ for short.
\mb{If we introduce the adjoint state $\p \in \V$ associated with $\y \in \V$ by 
\begin{align*}
    B(\p,\phii) = (\curl\y - \y ^d, \curl\phii)_{0,\Omega} \quad \forall \ \phii \in \V\,,
\end{align*}
\eqref{variin} can be equivalently written as the} following optimality system in terms of the state and the adjoint state, as well as the control variable:
\begin{subequations}
  \begin{empheq}[left=\empheqlbrace]{alignat = 2}
&B(\y ^*,\phii) = (\f+ \u^*,\phii)_{0,\Omega} \quad &\forall& \ \phii \in \V \,, \label{kkt1:1}\\
&B(\p ^*,\phii) = (\curl\y ^*-\y ^d,\curl\phii)_{0,\Omega} \quad &\forall& \ \phii \in \V\,,  \label{kkt1:2}\\
&(\p ^*+\alpha(\u^*-\u^d),\u - \u ^*)_{0,\Omega}\ge \0 \quad &\forall& \ \u \in \U^{ad}\,. \label{kkt1:3}
  \end{empheq}
\end{subequations}
If we further define the \mb{Lagrange multiplier (or the adjoint control)} for the variational inequality \eqref{kkt1:3}:
\begin{equation} \label{eq:confrederi}
    - \llambda^* = - \p^* - \alpha (\u ^* -\u ^d)\,,
\end{equation}
then \eqref{kkt1:3} can be reformulated in a compact form:
    \begin{equation} \label{kkt1:3_2}
        -\llambda^* \in \partial I_{\U^{ad}}(\u^*)\,,
    \end{equation}
where $\partial I_{\U^{ad}}$ is the subdifferential of the indicator function of the admissible set $\U^{ad}$ \cite{hiriart2013convex}. We remark that  
$\llambda^*$ is actually the Fr\'{e}chet derivative of the functional $J(\u)$.
We then conclude from \eqref{kkt1:3}, or \eqref{kkt1:3_2}, that $\u^*$ is nothing else than the $L^2$-projection of $-\frac{\p^*}{\alpha}+\u^d$
on $\U^{ad}$, i.e.,
\begin{equation} \label{eq:conproj}
\u^* = \mathbb{P}_{\U^{ad}}\left(-\frac{\p^*}{\alpha}+\u^d\right).
\end{equation}
Here and throughout this work, we denote by $\mathbb{P}_E$ the $L^2$-projection on the  convex subset $E$ of $\textit{\textbf{L}}^2(\Omega)$. It is worth mentioning that $\mathbb{P}_{\U^{ad}}$ in \eqref{eq:conproj} can also be understood in the pointwise sense due to the unilateral form of $\U^{ad}$:
\begin{equation*}
    \u^* = \max\left\{\0,-\frac{\p^*}{\alpha}+\u^d\right\}\quad \text{a.e. in}\ \Omega\,,
\end{equation*}
where the maximum is taken componentwise. We now end our discussion on the continuous optimal control problem with some regularity results. 
\mb{
\begin{prop} \label{prop:reg}
Assume that $\sigma$ is constant and $\Omega$ is convex or of class $C^{1,1}$. If $\u^d \in \textit{\textbf{H}}^1(\Omega)$ and $f \in \textit{\textbf{H}}(\div,\Omega)$,  then the triplet $(\y^*,\p^*,\u^*)$ satisfying the optimality system \eqref{kkt1:1}--\eqref{kkt1:3} has the  regularity:
$
    \u^*, \y^*,\p^* \in \textit{\textbf{H}}^1(\Omega)\,.
$
If we further assume that $\mu$ is constant, we have $\y^*,\p^* \in \textit{\textbf{H}}^1(\curl,\Omega)$. 


\end{prop}

\begin{proof}
  It is clear that \eqref{kkt1:2} is equivalent to the following equation:
  \begin{align} \label{auxeq_1}
      \curl \mu^{-1} \curl \p^* + \sigma \p^* = \curl (\curl \y^* - \y^d) \quad \text{in}\ \Omega \,,\quad \n \times \p^*   = 0 \quad   \text{on} \ \partial \Om\,,
  \end{align} 
  in the distributional sense. By taking the divergence on both sides of \eqref{auxeq_1}, we have $\div(\sigma \p^*) = 0$. Recall that the space 
  $\textit{\textbf{H}}_0(\curl,\Omega)\cap \textit{\textbf{H}}(\div,\Omega)$ is continuously imbedded in $\textit{\textbf{H}}^1(\Omega)$, if $\Omega$ is convex or of class $C^{1,1}$ (cf.\,\cite{amrouche1998vector}). Since $\sigma$ is assumed to be constant, we readily see $\p^* \in \textit{\textbf{H}}_0(\curl,\Omega)\cap \textit{\textbf{H}}(\div,\Omega) \hookrightarrow \textit{\textbf{H}}^1(\Omega)$. Noting that $\mathbb{P}_{\U^{ad}}$ is a continuous mapping from $\textit{\textbf{H}}^1(\Omega)$ to $\textit{\textbf{H}}^1(\Omega)$ (cf.\,\cite[Sec.\,5.10]{evans1998partial}\cite{kinderlehrer1980introduction}), we obtain from \eqref{eq:conproj} and $\u^d \in \textit{\textbf{H}}^1(\Omega)$ that $u^* \in \textit{\textbf{H}}^1(\Omega)$. Then using the fact that $\f + \u^* \in \textit{\textbf{H}}(\div,\Omega)$, 
  a similar argument applying to $\y^*$ yields $\y^* \in  \textit{\textbf{H}}^1(\Omega)$. 

If we further assume that $\mu$ is constant, it follows from the equation \eqref{kkt1:1} and de Rham diagram \cite[(3.60)]{monk2003finite} that $\curl \y^*$ belongs to the space $\textit{\textbf{H}}(\curl,\Omega)\cap \textit{\textbf{H}}_0(\div,\Omega)$, which is also continuously imbedded in $\textit{\textbf{H}}^1(\Omega)$. Hence, there holds $ \curl \y^* \in \textit{\textbf{H}}^1(\Omega)$, which gives $\y^* \in \textit{\textbf{H}}^1(\curl,\Omega)$ (see \eqref{defhsspace}). Similarly, we can conclude by \eqref{kkt1:2} that $\p^* \in \textit{\textbf{H}}^1(\curl,\Omega)$, under the assumption \eqref{asspdata} that $\y^d \in \textit{\textbf{H}}_0(\curl,\Omega)$. The proof is complete. 
\end{proof}
\begin{re}
The above result essentially relies on the regularity theory for Maxwell's equations (cf.\,\cite{weber1981regularity}\cite{yin2004regularity}\cite{alberti2014elliptic}\cite{alberti2018lectures}).  Here, we consider the case where the coefficients are isotropic and constant, so that the proof is direct. By using the recent result \cite[Theorem 1]{alberti2014elliptic}, one can similarly obtain the following generalization. Assume that $\mu$ is constant, $\sigma \in W^{1,3+\delta}(\Omega)$ for some $\delta > 0$  and $\Omega$ is of class $C^{1,1}$. If $\u^d \in \textit{\textbf{H}}^1(\Omega)$ and $f \in \textit{\textbf{H}}(\div,\Omega)$ hold as above, then the solution $(\y^*,\p^*,\u^*)$ to the optimality system \eqref{kkt1:1}--\eqref{kkt1:3} still satisfies  
$
    \u^*\in \textit{\textbf{H}}^1(\Omega)
$
and $\y^*,\p^* \in \textit{\textbf{H}}^1(\curl,\Omega)$. 
\end{re} 
}

From the above discussion about the regularity, we can see that if there 
are no additional assumptions on the physical coefficients, the given data and 
the domain $\Om$,  the optimal control $\u^*$ should generally be of low regularity (less regular than $\textit{\textbf{H}}^1$). 
\mb{Hence we may expect from the standard ($h$-version) FEM theory that on uniform meshes}, the simplest piecewise constant approximation can already achieve 
the optimal approximation accuracy,
and any higher-order approximations can not improve the numerical accuracy of the optimal control. 
In fact, this is still true even if the optimal control $\u^*$ 
has the $\textit{\textbf{H}}^1$-regularity, since we aim only at the $L^2$-error of the control variable.
This is one of the main motivations of the current work. \mb{We would like to further remark that if the solution is piecewise analytic (singularity 
is allowed on a lower dimensional manifold), one may use general versions of FEM (e.g. $hp$-FEM) to get better convergence rate, even the exponential rate \cite{szabo1991finite}\cite{babuvska1994p}\cite{schwab1998p}\cite{suri2001p} (but the concrete design and implementation of such algorithms would be typically difficult and 
beyond the scope of this work).}

The rest of this section is devoted to the finite element 
discretization of the control problem. We consider a family of conforming and shape regular triangulations $\{\mathscr{T}_h\}$ of $\Omega$, which coincide with the partition $\Omega = \prod_{i=1}^m \Omega_i$ in the sense that $\Omega_i = \bigcup_{\sss T\subset \Omega_i} T$. We set $h := \max_{T \in \mathscr{T}_h}h_T$, where $h_T$ denotes the diameter of $T$. We also refer to $\mathscr{F}_h:= \bigcup \{\partial T\cap \Omega\,;\ T \in \mathscr{T}_h\}$, the set of all the interior faces, as the skeleton of the triangulation $\mathscr{T}_h$. The  diameter of a face $F$ is denoted by $h_F$.
In our algorithm, we use the lowest-order $\textit{\textbf{H}}(\curl)$-conforming edge element space of N\'{e}d\'{e}lec's first family with zero tangential trace: 
\begin{equation*}
\V_h = \left\{ \v_h \in \textit{\textbf{H}}_0(\curl,\Om)\,; \  \v_h|_T = \a_T \times \x + \b_T \ \, \text{with}\ \,\a_T , \b_T \in \R^3\,,\ T \in \mathscr{T}_h \right\}\,,
\end{equation*}
to approximate both the state and the adjoint state.
For approximating the control variable, we employ the simplest space of
element-wise constant functions with respect to the triangulation $\mathscr{T}_h$:
\begin{equation*}
\U_h = \left\{\u_h\in \textit{\textbf{L}}^2\left(\Om\right)\,;\  \u_h |_T \in P_0(T)\,, \ T\in \mathscr{T}_h \right\}\,.
\end{equation*}
Here and in what follows, $P_k(T)$ is \mb{the space of polynomials of degree $\le k$ over $T$}. The discrete admissible set $\U_h^{ad}$ is then given by
\begin{equation*}
\U_h^{ad} := \U_h \cap \U^{ad} = \{\u_h \in \U_h \,;\  \u_h \ge \0\}.
\end{equation*} 
With the help of these notions, we formulate the discrete optimal control problem:
\begin{align}
\mbox{minimize}~~ & J(\y_h,\u_h)
=\frac{1}{2}\norm{\curl \y_h-\y^d}_{0,\Omega}^2+\frac{\alpha}{2}\norm{\u_h-\u^d}_{0,\Omega}^2 \label{discrete problem1}\\
\mbox{over}~~ & (\y_h,\u_h) \in \V_h \times \U_h^{ad}\notag \\
\mbox{subject to}~~ & B(\y_h , \phii_h) = (\f + \u_h ,\phii_h)_{0,\Omega} \quad \forall \ \phii_h \in \V_h\,. \label{discrete problem2}
\end{align}
As in the continuous case, the existence and uniqueness of the solution to the discrete system 
\eqref{discrete problem1}-\eqref{discrete problem2} can be guaranteed,
and the corresponding
optimality system reads as follows: \mb{find $(\y^*_h, \p^*_h, \u_h^*) \in \V_h \times \V_h \times \U^{ad}_h$ by solving}
\begin{subequations}
  \begin{empheq}[left=\empheqlbrace]{alignat = 2}
&B(\y ^*_h,\phii_h) = (\f+ \u^*_h,\phii_h)_{0,\Omega} \quad &\forall& \ \phii_h \in \V_h\,, \label{kkt2:1}\\
&B(\p ^*_h,\phii_h) = (\curl\y ^*_h-\y ^d,\curl\phii_h)_{0,\Omega} \quad &\forall& \ \phii_h \in \V_h\,, \label{kkt2:2}\\
&(\mathbb{P}_h\p ^*_h+\alpha (\u^*_h-\u^d_h),\u_h - \u ^*_h)_{0,\Omega} \ge 0 \quad &\forall& \ \u_h \in \U^{ad}_h\,,\label{kkt2:3}
  \end{empheq}
\end{subequations}
where $\mathbb{P}_h$ denotes the $L^2$-projection $\mathbb{P}_{\U_h}$:
\begin{equation}\label{averoper}
(\mathbb{P}_h \v)_T := \frac{1}{|T|}\int _T \v(\x)d\x: \mb{\textit{\textbf{L}}^2(\Omega) \to \U_h},
\end{equation}
and the approximate target control $\u_h^d$ in \eqref{kkt2:3} is given by $\mathbb{P}_h \u^d$.
Similarly, we denote by $\y_h(\u_h)$ the solution to \eqref{discrete problem2} and introduce the discrete reduced cost functional: $$J_h(\u):= \frac{1}{2}\norm{\curl \y_h(\u_h)-\y^d}_{0,\Om}^2+\frac{\alpha}{2}\norm{\u_h-\u^d}_{0,\Om}^2\,,$$
then its Fr\'{e}chet derivative is given by \mb{(one can actually show $\llambda_h^* \in \U_h^{ad}$, i.e., $\llambda_h^* \ge \0$; see \eqref{eq:disproj_2})}
\begin{equation} \label{eq:disfrederi}
    \llambda_h^* = \mathbb{P}_h\p ^*_h+\alpha (\u^*_h-\u^d_h) \mb{\in \U_h}\,.
\end{equation}
We also see from \eqref{kkt2:3} and the relation  $\mathbb{P}_{\U_h^{ad}} = \mathbb{P}_{\U^{ad}} \mathbb{P}_{h} $ that
\begin{equation}  \label{eq:disproj}
    \u_h^* = \mathbb{P}_{\U_h^{ad}}\left(-\frac{\p_h^*}{\alpha}+\u^d\right) = \mathbb{P}_{\U^{ad}}\left(-\frac{\mathbb{P}_h\p_h^*}{\alpha}+\u_h^d\right),
\end{equation}
which directly implies a pointwise representation for the discrete optimal control:
\begin{equation} \label{eq:disproj_2}
    \u_h^* = \max\left\{\0, -\frac{\mathbb{P}_h\p ^*_h}{\alpha}+\u_h^d  \right\}\,.
\end{equation}

\section{Adaptive algorithm and the a posteriori error analysis} \label{algorithm}
 In this section, we give a complete discussion of the adaptive edge element method for solving the $\textit{\textbf{H}}(\curl)$-elliptic control problem  and derive the a posteriori error estimates. An adaptive finite element method typically takes the successive loops:
\begin{equation*}
    \textbf{SOLVE}\rightarrow \textbf{ESTIMATE}\rightarrow \textbf{MARK}\rightarrow \textbf{REFINE}\,,
\end{equation*}
where the a posteriori error estimation (module {\bf ESTIMATE}) is the core step in the design of an adaptive algorithm, through which we can extract the information on the error distribution. For our algorithm and analysis, we shall use the residual-type a posteriori error estimators in terms of numerically computable quantities, which include the element residuals $\eta_T$, $T \in \T{h}$, and the face residuals 
$\eta_F$, $F \in \mathscr{F}_h$:
\begin{align} \label{eq:elefacres}
\eta_T^2 := \sum_{i=1,2}(\eta_{y,T}^{(i)})^2 +\sum_{i=1,2,3} (\eta_{p,T}^{(i)})^2, \quad \eta_F^2 := \sum_{i=1,2}(\eta_{y,F}^{(i)})^2 + (\eta_{p,F}^{(i)})^2\,,
\end{align}
where $\eta_{y,T}^{(i)}$ and $\eta_{p.T}^{(i)}$ are the element residuals defined by 
\begin{align*}
&\eta_{y,T}^{(1)} := h_T \norm{\f + \u_h^* -\curl\mu^{-1} \curl \y_h^* - \sigma \y_h^*}_{0,T}\,, \quad \eta_{y,T}^{(2)} := h_T \norm{\div(\f -\sigma\y_h^*)}_{0,T}\,,
\end{align*}
and 
\begin{align*}
&\eta_{p,T}^{(1)} := h_T \norm{\curl \y^d+\curl \mu^{-1} \curl\p_h ^*+\sigma \p_h^*}_{0,T}\,,\quad \eta_{p,T}^{(2)} := h_T \norm{\div(\sigma \p_h^*)}_{0,T}, \quad \eta_{p,T}^{(3)} := \norm{\p_h^*-\mathbb{P}_h\p_h^*}_{0,T}\,,
\end{align*}
while $\eta_{y,F}^{(i)}$ and $\eta_{p,F}^{(i)}$ are the face residuals defined by
\begin{align*}
&\eta_{y,F}^{(1)}:=h_F^{1/2}\norm{[\gamma_t(\mu^{-1} \curl \y_h^*)]_F}_{0,F}\,, \qquad \qquad \qquad \eta_{y,F}^{(2)}:=h_F^{1/2}\norm{[\gamma_n(\f + \u_h^* -\sigma \y_h^*)]_F}_{0,F} \,, \\
&\eta_{p,F}^{(1)}:=h_F^{1/2}\norm{[\gamma_t(-\mu^{-1} \curl \p_h^*+\curl\y_h^*)]_F}_{0,F}\,, \quad \eta_{p,F}^{(2)}:=h_F^{1/2}\norm{[\gamma_n(\sigma \p_h^*)]_F}_{0,F}\,.
\end{align*}
Here $[\cdot]_F$ stands for the jump across the face $F$.
\mb{We remark that it may not be easy to see how to define these terms at the first glance, but we shall see that  they are generated naturally in the reliability analysis.} For ease of exposition,  we denote by
\begin{equation*}
    \eta^2_h(T) = \eta_T^2 + \frac{1}{2} \sum_{F \in \partial T \cap\Omega}\eta_F^2\,,
\end{equation*}
the residual-type error indicator associated with an element $T$. We should note that there are no face residuals on the boundary of domain since the state equation satisfies the homogeneous boundary condition. For the a posteriori error estimates and the convergence analysis, a lower-order data oscillation related to $\u^d$ is needed:
\begin{align*}
    \osc^2_h(\u ^d):=\sum_{T \in \mathscr{T}_h} \osc_{T}^2(\u^d)\,, \quad  \osc_{T}(\u^d):=\norm{\u^d-\u^d_h}_{0,T}\,, \ T \in \mathscr{T}_h.
\end{align*}
Some higher-order data oscillations  associated with $\y^d$ and $\f$ shall also be involved: 
\begin{align*}
    \osc^2_h(\y^d):=\sum_{T \in \mathscr{T}_h} \osc_{T}^2(\y^d)\quad \text{with} \quad \osc_{T}(\y ^d):=h_T\norm{\curl(\y ^d-\y_h^d)}_{0,T}\,,\  T \in \mathscr{T}_h\,,   
\end{align*}
and 
\begin{align*}
    \osc^2_h(\f):= \sum_{T \in \mathscr{T}_h} \osc_{T}^2(\f)    
\end{align*}
with 
\begin{align*}
\osc_{T}(\f):= h_T\norm{\f-\f_h}_{\div,T} + \sum_{F \in \partial T\cap \Omega}h_F^{1/2}\norm{[\gamma_n(\f-\f_h)]_F}_{0,F}\,,\  T \in \mathscr{T}_h\,,
\end{align*}
where we assume that $\y_h^d \in \V_h$ and $\f_h \in \U_h$ are some approximations of $\y^d$ and $\f$, respectively. In contrast to the element residuals and face residuals associated with the discrete solutions, the data oscillation
$\osc_h(\u ^d)$ is typically of lower order for a non-smooth target control $\u^d$, and at most of $O(h)$, by the Poincar\'{e} inequality, even if $\u^d$ has a certain regularity.  Meanwhile, the data oscillations $\osc_h(\y^d)$ and $\osc_h(\f)$ are of the same order as the residuals, and shall be of higher order if the data
$\y^d$ and $\f$ have additional regularities. 
Therefore, we could replace $\y^d$ and $\f$ in the element residuals $\eta_{y,T}^{(i)},\ i = 1, 2$ and $\eta_{p,T}^{(1)}$, as well as some face residuals, with $\y_h^d$ and $\f_h$ for ease of implementation without any influence on the performance of the algorithm and any essential change of the analysis.  

As we shall see, since the inconsistent discrete spaces are used and no additional regularity assumptions are added on the data, the lower-order data oscillations may have a significant contribution to the total error so that a single residual-type error estimator $\eta_h$ is not enough to capture the error distribution accurately. Therefore, to design an efficient adaptive algorithm for our problem, it is necessary to take into account the data oscillations \cite{gaevskaya2007convergence,hintermuller2008posteriori}. 
In view of this, we introduce a new mixed error indicator, by incorporating the lower-order data oscillation $\osc_T(\u^d)$ into the  residual-type error estimator $\eta_h$,
\begin{equation} \label{eq:toterrind}
    \h{\eta}_h^2 = \sum_{T \in \mathscr{T}_h} \h{\eta}^2_h(T)   \quad \text{with} \quad \h{\eta}^2_h(T) =  \eta^2_h(T) + \osc^2_T(\u^d)\,.
\end{equation}
We are now in a position to present the algorithm (see Algorithm \ref{alg:Framwork} below), based on the error estimator $\h{\eta}_h$. \mb{In what follows}, we shall use the iteration index $k$ to indicate the dependence of a variable or a quantity on a  particular mesh generated in the adaptive process, \mb{for instance}, we write $\eta_k(T)$ 
for $\eta_h(T)$
to emphasize that we are considering the error estimator $\eta_h(T)$ on the mesh $\T{k}$. 
\begin{algorithm}
\caption{ Adaptive edge element method }
\label{alg:Framwork}
\begin{algorithmic}[1] 
\STATE Specify a shape regular initial mesh $\mathscr{T}_0$ on $\Omega$ and set the iteration index $k:=0$.
\label{code:initial}
\STATE ({\bf SOLVE}) Compute the numerical solution $(\yk,\pk\,\uk)$ to the discrete optimality system \eqref{kkt2:1}-\eqref{kkt2:3} on the mesh $\mathscr{T}_k$.
\label{solve}
\STATE ({\bf ESTIMATE}) Compute the error estimator $\h{\eta}_k(T)$ defined in \eqref{eq:toterrind} 
for each element $T \in \T{k}$.\label{estimate}

\STATE ({\bf MARK}) Mark a subset $\mathscr{M}_k \subset \mathscr{T}_k$ containing 
at least one element $\widetilde{T}_k$ that satisfies
\begin{equation} \label{eq:algocond}
    \h{\eta}_k(\widetilde{T}_k) = \max_{T\in\mathscr{T}_k} \h{\eta}_k(T)\,.
\end{equation} 
\label{mark}
\STATE ({\bf REFINE}) Refine elements in $\mathscr{M}_k$
and \mb{other necessary elements by bisection 
to generate the smallest conforming mesh $\mathscr{T}_{k+1}$ with $\mathscr{T}_{k+1} \cap \mathscr{M}_k = \varnothing$}. \label{refine}
\STATE Set $k = k + 1$ and go to Step $2$ until the preset stopping criterion is met.
\label{code:fram:stop}
\end{algorithmic}
\end{algorithm}

Several further remarks concerning each module in Algorithm \ref{alg:Framwork} are in order. First, in the module {\bf SOLVE}, we are required to solve a large-scale quadratic optimization problem with discretized PDE-constraints efficiently. It is nowadays a very active research area, and many high-performance algorithms and preconditioners have been developed for this purpose (cf.\,\cite{pearson2012new}\cite{pearson2012regularization}\cite{pearson2017fast}\cite{stoll2015low}\cite{bunger2020low}). 

Second, to guarantee the convergence of the algorithm, we only need the condition \eqref{eq:algocond} in the module {\bf MARK} 
that the marked elements contain an element with the largest error indicator, which can be met by 
many popular marking strategies, such as the maximum strategy \cite{babuvvska1978error}, the equidistribution strategy \cite{eriksson1991adaptive} and the D\"{o}rfler's strategy \cite{dorfler_convergent_1996}.
In our numerical experiments (see Section \ref{experiment}), we shall use the D\"{o}rfler's strategy to select the marked elements 
$\mathscr{M}_k$ with minimal cardinality such that for a given $\theta \in (0,1)$, there holds
\begin{align}
\sum\limits_{T\in \mathscr{M}_k} \h{\eta}^2_k(T) \geq \theta \sum\limits_{T\in \mathscr{T}_k} \h{\eta}^2_k(T)\,. \label{mark1}
\end{align}

Third, we only include the data oscillation $\osc_{h}(\u^d)$ in the definition of $\h{\eta}_h$ since it is 
a lower-order term that may provide a dominant error contribution among all the data oscillations. On the other hand, the behavior of the optimal control $\u^*$ relies largely on the properties of 
the obstacle function $\boldsymbol{\psi}$, 
whose information is further transferred to $\u^d$. We hence expect that the approximation ability of $\U_h$ associated with the current mesh $\T{h}$ for $\u^d$ can directly influence
the performance of our AFEM; see also \cite{hintermuller2008posteriori}\cite{gaevskaya2007convergence} for related discussions.
In the case where $\boldsymbol{\psi}$ is a constant function and $\u^d = 0$, $\osc_h(\u^d)$
vanishes and $\h{\eta}_h$ becomes the standard residual-type error estimator $\eta_h$.

Finally, for the module {\bf REFINE}, \mb{all elements of $\mathscr{M}_k$ are bisected at least once}, and some additional elements in $\mathscr{T}_k \backslash \mathscr{M}_k$ may also need to be subdivided in order to generate a sequence of uniformly shape regular and conforming meshes $\{\mathscr{T}_k\}_{k \ge 0}$;
see \cite{cascon2008quasi} and \cite[Section 4]{nochetto2009theory}
for the detailed mesh refinement algorithm and the necessary assumptions on the initial triangulation $\T{0}$.
Such a refinement process ensures that all the generic constants involved in the inequalities below depend only on the shape regularity of the initial mesh and the given data.

We shall next present the a posteriori error analysis based on the error estimator $\h{\eta}_h$ \eqref{eq:toterrind}, including both the reliability and efficiency estimates, which is similar in spirit to the ones given in \cite{hoppe2015adaptive}\cite{schoberl2008posteriori}, but with several 
main difficulties and differences as stated in the introduction, especially those caused by the inconsistency between the discrete spaces of the state and control.


\subsection{Reliability}
In this section, we show that the error estimator $\h{\eta}_h$ is reliable in the sense that it can provide an upper bound for the total error between the true solution and the numerical solution: 
$$
\norm{\y_h^*-\y^*}_{\curl,\Om}+\norm{\p^*_h-\p^*}_{\curl,\Om}+\norm
{\u^*_h-\u^*}_{0,\Om}\,.
$$
For this purpose, \mb{following \cite{schoberl2008posteriori}}, we introduce a helpful quasi-interpolation operator $\Pi_h$. \mb{We start with the definition} of the extended neighborhood $\widetilde{\Om}_T$ for an element $T\in \mathscr{T}_h$ and fix some notations. Recall that the neighborhood $\Omega_\v$ of a vortex \v \  is defined as the union of all the elements that contain the vertex \v,\  and the extended neighborhood of a vertex $\v$ is given by  $\widetilde{\Om}_\v:=\bigcup_{\v'\in \Om_\v}\Om_{\v'}$. We define the extended neighborhood of an element $T$ by $\widetilde{\Om}_T=\bigcup_{\v\in T}\widetilde{\Om}_\v$. \mb{An important consequence of the uniformly shape regularity of $\{\mathscr{T}_h\}$ is that the cardinality of $\widetilde{\Om}_T$ is uniformly bounded (cf.\,\cite[Section 4.3]{nochetto2009theory}): 
\begin{align}  \label{eq:localuni}
    \max_{T \in \mathscr{T}_h} \# \widetilde{\Om}_T \le C(\mathscr{T}_0)\,.
\end{align} 
The corresponding converse fact is that the collection of extended neighborhoods $\widetilde{\Omega}_T,\  T \in \mathscr{T}_h$, covers each element in $\mathscr{T}_h$ finite times uniformly:  
\begin{align} \label{eq:localuni2}
    \max_{T \in \mathscr{T}_h} \# \{T' \in \mathscr{T}_h\,;\ T \in \widetilde{\Om}_{T'}\} \le C(\mathscr{T}_0)\,.
\end{align}
Here the constants $C(\mathscr{T}_0)$ only depend on the initial mesh $\mathscr{T}_0$.}
\begin{lem}{\cite[Theorem 1]{schoberl2008posteriori}}\label{interpo}
 There exists an interpolation operator $\Pi_h:\textit{\textbf{H}}_0(\curl,\Om)\rightarrow \V_h$ such that for any $\u \in \textit{\textbf{H}}_0(\curl,\Om)$, $\u-\Pi_h\u $ has the decomposition:
 \begin{equation*}
 \u-\Pi_h\u = \nabla \varphi + \z\,,
 \end{equation*}
 with $\varphi \in H_0^1(\Om),\z \in \textit{\textbf{H}}^1_0(\Om)$. \mb{Moreover, the following estimates hold}
 \begin{align}
 h_T^{-1}\norm{\varphi}_{0,T}+\norm{\nabla \varphi}_{0,T} &\lesssim \norm{\u}_{0,\widetilde{\Om}_T}\,, \label{inp:es1}\\
 h_T^{-1}\norm{\z}_{0,T}+\norm{\nabla \z}_{0,T} &\lesssim \norm{\curl\u}_{0,\widetilde{\Om}_T}\,. \label{inp:es2}
 \end{align}
\end{lem}

Since the optimal state $\y^*$ and adjoint state $\p^*$ satisfy the coupled system \eqref{kkt1:1}-\eqref{kkt1:3}, the Galerkin orthogonality, which is important for the a posteriori error analysis for the linear boundary value problems, does not hold here. To compensate it, we introduce the intermediate state and adjoint state, $\y(\u_h^*)$ and $\p(\u_h^*)$, by the equations:
\begin{align}
&B(\y(\u_h^*),\phii)=(\f+\u_h^*,\phii)  \qquad \qquad  \quad \quad \ \ \,  \forall \ \phii \in \V\,, \label{intermi:1}\\
&B(\p(\u_h^*),\boldsymbol {\psi})=(\curl\y(\u_h^*)-\y^d,\curl\boldsymbol {\psi})  \quad \forall \ \boldsymbol {\psi} \in \V\,. \label{intermi:2}
\end{align}
If we use the variational discretization for the control variable, we can derive the following error equivalence (cf.\,\cite{gong2017adaptive}\cite{hinze2005variational}):
\begin{equation*}
    \norm{\y(\u_h^*)-\y_h^*}_{\curl,\Om}+\norm{\p(\u_h^*)-\p_h^*}_{\curl,\Om} \approx \norm{\y_h^*-\y^*}_{\curl,\Om}+\norm{\p^*_h-\p^*}_{\curl,\Om}+\norm{\u^*_h-\u^*}_{0,\Om}\,,
\end{equation*}
which allows us to directly conclude the reliability of the error estimator from the known result concerning Maxwell's equations \cite[Corollary 2]{schoberl2008posteriori}. If we use the edge element discretization for the control as in \cite{hoppe2015adaptive}, the above error equivalence still holds, up to the data oscillation $\osc_h(\u^d)$. However, this is not the case in our algorithm since the discrete spaces of the control and state variables are different. Instead, we have the following result.
\begin{lem} \label{lem:forreal}
    Let the triplets $(\y^*,\p^*,\u^*)$ and $(\y_h^*,\p_h^*,\u_h^*)$ be the solutions to \eqref{kkt1:1}-\eqref{kkt1:3} and
    \eqref{kkt2:1}-\eqref{kkt2:3}, respectively. Then it holds that
    \begin{align*}
        &\norm{\y_h^*-\y^*}_{\curl,\Om} + \norm{\p^*_h-\p^*}_{\curl,\Om} +\norm{\u^*_h-\u^*}_{0,\Om}  \\
        \lesssim &\norm{\y(\u_h^*)-\y_h^*}_{\curl,\Om}+\norm{\p(\u_h^*)-\p_h^*}_{\curl,\Om} + \norm{\mathbb{P}_h\p_h^*-\p_h^*}_{0,\Om} + \osc_h(\u^d)\,.
    \end{align*}
\end{lem}
\begin{proof}
    By the well-posedness of the $\textit{\textbf{H}}(\curl)$-elliptic variational problem, we have
\begin{align}
    &\norm{\y ^*-\y(\u_h^*)}_{\curl,\Om}\lesssim\norm{\u-\u^*_h}_{0,\Om}\,, \label{eq:estiner1}\\
    &\norm{\p^*-\p(\u^*_h)}_{\curl,\Om}\lesssim\norm{\curl\y^*-\curl\y(\u^*_h)}_{0,\Om}\lesssim \norm{\u-\u^*_h}_{0,\Om}\,, \label{eq:estiner2}
\end{align}
which, combined with the \mb{triangle inequality}, reduce the proof of the lemma to the estimate of $\norm{\u^*_h-\u^*}_{0,\Om} $.
In view of \eqref{eq:conproj} and \eqref{eq:disproj}, we get
\begin{equation} \label{eq:erruc}
    \norm{\u^* - \u_h^*}^2_{0,\Omega} \le \Big(\u^* - \u^*_h, -\frac{\p^*-\mathbb{P}_h \p_h^*}{\alpha} + \u^d - \u_h^d\Big)_{0,\Omega}\,,
\end{equation}
by the contraction property of $L^2$-projections \cite[Proposition 5.3]{brezis2010functional}. Moreover,
we can deduce, by taking $\phii = \p^*-\p(\u_h^*)$ in (\ref{intermi:1}) and $\boldsymbol {\psi} = \y^* - \y(\u_h^*)$ in (\ref{intermi:2}), that 
\begin{equation}\label{relation}
(\u^*-\u_h^*,\p^*-\p(\u_h^*))_{0,\Omega} =\norm{\curl\y^*-\curl\y(\u_h^*)}_{0,\Omega}^2 \geq 0\,.
\end{equation}
Combining \eqref{relation} with \eqref{eq:erruc} helps us obtain
\begin{align*}
    \norm{\u_h^*-\u^*}_{0,\Om}\lesssim \norm{\mathbb{P}_h\p_h^*-\p_h^*}_{0,\Om} + \norm{\p_h^*-\p(\u_h^*)}_{0,\Om} +  \norm{\u^d_h-\u^d}_{0,\Om}\,,
   \end{align*}
which completes the proof of the lemma.
\end{proof}
With the above preparations, we are now ready to prove the reliability of the error estimator. 

%
%
\begin{thm}\label{reliability}
    Let the triplets $(\y^*,\p^*,\u^*)$ and $(\y_h^*,\p_h^*,\u_h^*)$ be the solutions to the continuous and discrete optimality systems \eqref{kkt1:1}-\eqref{kkt1:3} and
    \eqref{kkt2:1}-\eqref{kkt2:3}, respectively. Then we have the following reliability estimate:
 \begin{equation}\label{eq:reliability}
 \norm{\y_h^*-\y^*}_{\curl,\Om}+\norm{\p^*_h-\p^*}_{\curl,\Om}+\norm{\u^*_h-\u^*}_{0,\Om}
    \lesssim \h{\eta}_h\,.
 \end{equation}
\end{thm}

\begin{proof}
    By Lemma\,\ref{lem:forreal}, it suffices to estimate $\norm{\y(\u_h^*)-\y_h^*}_{\curl,\Om}+\norm{\p(\u_h^*)-\p_h^*}_{\curl,\Om}$ to obtain the reliability estimate \eqref{eq:reliability}.
    We first consider the estimate for the state variable $\y$. Let $\e_{\y}$ be $ \y(\u_h^*)-\y_h^*$, and recall the norm 
    equivalence:  $\norm{\v}_B
    \approx \norm{\v}_{\curl,\Om}$. We can derive,
    by the Galerkin orthogonality,
\begin{align} \label{eq:mid_pro}
\norm{\e_{\y}}_{\curl,\Om}^2 &\approx B(\e_{\y} ,\e_{\y}-\Pi_h \e_{\y}) =(\f+\u_h^*-\sigma\y_h^*,\e_{\y}-\Pi_h \e_{\y})_{0,\Omega} - (\mu^{-1} \curl\y_h^*,\curl(\e_{\y}-\Pi_h \e_{\y}))_{0,\Omega}\,. 
\end{align}
A direct application of Lemma \ref{interpo} gives us the decomposition: $ \e_{\y}-\Pi_h \e_{\y} = \nabla \varphi + \z$ with $\varphi \in H_0^1(\Om),\z \in \textit{\textbf{H}}^1_0(\Om)$. Substituting it into \eqref{eq:mid_pro} and using  integration by parts for each $T$, we have 
\begin{align}\label{reab:1}
&(\f+\u_h^*-\sigma\y_h^*,\nabla \varphi + \z)_{0,\Omega} -(\mu^{-1} \curl\y_h^*,\curl(\nabla \varphi +\z))_{0,\Omega} \notag \\
=&\sum_{T\in\mathscr{T}_h}-(\div \f ,\varphi)_{0,T}+(\f+\u_h^*-\sigma\y_h^*-\curl \mu^{-1} \curl \y^*_h,\z)_{0,T}+(\div(\sigma \y_h^*),\varphi)_{0,T} \notag \\
& +\sum_{ F \in \mathscr{F}_h}([\gamma_n(\f+\u^*_h-\sigma\y_h^*)]_F,\varphi)_{0,F}+([\gamma_t(\mu^{-1} \curl \y_h^*)]_F,\z)_{0,F}\,.
\end{align}
To proceed, \mb{by the scaled trace inequality \cite[Corollary 6.1]{nochetto2009theory}:}
\begin{equation} \label{auxeq_trace}
\mb{\norm{w}_{0,F} \lesssim h_F^{-1/2}\norm{w}_{0,T} + h_F^{1/2} \norm{\nabla w}_{0,T} \quad \text{for} \ F \in \partial T,\ w \in H^1(T)\,,}
\end{equation}
and estimates (\ref{inp:es1}) and (\ref{inp:es2}), we obtain from \eqref{eq:mid_pro} and \eqref{reab:1} and the definitions of the error estimators, 
\begin{equation}\label{reab:2}
\norm{\e_{\y}}_{\curl,\Om}^2\lesssim\sum_{T \in \mathscr{T}_h}\left(\eta_{y,T}^{(1)}+\eta_{y,T}^{(2)}\right)\norm{\e_{\y}}_{\curl,\widetilde{\Om}_T}
+ \sum_{T \in \mathscr{T}_h} \sum_{F \in \partial T \cap \Omega}\left(\eta_{y,F}^{(1)}+\eta_{y,F}^{(2)}\right)\norm{\e_{\y}}_{\curl,\widetilde{\Om}_T}\,.
\end{equation}
\mb{By the property \eqref{eq:localuni2},}
 the desired estimate follows from \eqref{reab:2} and the Cauchy's inequality:
\begin{equation}\label{esfory}
\norm{\y(\u^*_h)-\y_h^*}_{\curl,\Om}\lesssim \eta_h\,.
\end{equation}

The error $\e_{\p}:=\p(\u_h^*) - \p_h^*$ for the adjoint state can be analysed similarly. \mb{We note 
\begin{align*}
    \norm{\e_{\p}}_{\curl,\Om}^2 \approx & B(\e_{\p},\e_{\p}-\Pi_h \e_{\p})+B(\e_{\p},\Pi_h \e_{\p})\,,
\end{align*} 
and write $\e_{\p}-\Pi_h \e_{\p} = \nabla \varphi + \z$ by Lemma \ref{interpo}. 
Then some elementary calculations give us that 
\begin{align*}
    &B(\e_{\p},\e_{\p}-\Pi_h \e_{\p})  = B(\p(\u_h^*) - \p_h^*,\nabla \varphi + \z)\\
   \lesssim & (\curl\y_h^*-\y^d-\mu^{-1} \curl\p_h^*,\curl \z)_{0,\Omega} - (\sigma\p_h^*,\nabla \varphi +\z)_{0,\Omega}
    +\norm{\y_h^*-\y(\u_h^*)}_{\curl,\Om}\norm{\e_{\p}}_{\curl,\Om}\,.
\end{align*}
Moreover, by equations \eqref{kkt2:2} and \eqref{intermi:2}, we have  
\begin{align*}
    \left | B(\e_{\p},\Pi_h \e_{\p}) \right | \lesssim \norm{\y_h^*-\y(\u_h^*)}_{\curl,\Om}\norm{\e_{\p}}_{\curl,\Om}\,.
\end{align*}
Then we can derive by using integration by parts and the above estimates that}
\begin{align*}
\norm{\e_{\p}}_{\curl,\Om}^2 
\lesssim & (\curl\y_h^*-\y^d-\mu^{-1} \curl\p_h^*,\curl \z)_{0,\Omega}-(\sigma\p_h^*,\nabla \varphi +\z)_{0,\Omega}
+\norm{\y_h^*-\y(\u_h^*)}_{\curl,\Om}\norm{\e_{\p}}_{\curl,\Om} \notag \\
= &\sum_{T \in \mathscr{T}_h}(-\curl\y^d-\curl \mu^{-1} \curl \p_h^*-\sigma\p_h^*,\z)_{0,T}+(\div(\sigma \p_h^*),\varphi)_{0,T} \notag \\
&-\sum_{F \in \mathscr{F}_h}([\gamma_t(\curl\y_h^*-\mu^{-1} \curl \p^*_h)]_F,\z)_{0,F} - ([\gamma_n(\sigma\p_h^*)]_F,\varphi)_{0,F} \notag \\
&+\norm{\y_h^*-\y(\u_h^*)}_{\curl,\Om}\norm{\e_{\p}}_{\curl,\Om}\,,
\end{align*}
which, by \eqref{esfory}, the trace inequality \eqref{auxeq_trace} and 
Lemma \ref{interpo}, gives 
\begin{equation}\label{esforp}
\norm{\p(\u^*_h)-\p_h^*}_{\curl,\Om}\lesssim \eta_h\,.
\end{equation}
Combining estimates (\ref{esfory}) and (\ref{esforp}) with Lemma \ref{lem:forreal} and the definition of $\h{\eta}_h$ completes the proof of \eqref{eq:reliability}.
\end{proof}

\subsection{Efficiency}
In this section, we consider the efficiency estimate, which is another aim of the a posteriori error analysis. For this,  we need the so-called bubble functions, which plays a similar role to the cut-off functions and can help us estimate the local errors. As we shall see soon, 
when we deal with the divergence parts of the residual-type error estimator $\eta_h$ (i.e., $\eta^{\sss (2)}_{y,T}$, $\eta^{\sss (2)}_{p,T}$, 
$\eta^{\sss (2)}_{y,F}$ and $\eta^{\sss (2)}_{p,F}$), 
the higher-order bubble functions have to be used to ensure the vanishing boundary traces of some terms. Moreover, the $\curl$ structures in the right-hand sides of the adjoint equations \eqref{kkt1:2} and \eqref{kkt2:2} also need our special and careful treatment. 
These important points were not addressed in \cite{hoppe2015adaptive}.

For the reader's convenience, we next briefly review the definition of the bubble functions and some basic results, see \cite{verfurth1994posteriori} and \cite{ainsworth2011posteriori} for a comprehensive introduction of this topic. We define the bubble function for an element $T \in \T{h}$ by $b_T(\x) = 256\Pi_{i=1}^4\lambda_i^T(\x)$, $\x \in T$, where $\lambda_i^T, 1 \le i \le 4$, are the barycentric coordinate functions associated with four vertices of $T \in \mathscr{T}_h$. Similarly, the bubble function for a face $F \in \mathscr{F}_h$ is given by $b_F|_T(\x) =  27\Pi_{i=1}^3\lambda_i^F(\x)$, $\x \in T \in \omega_F$. Here, $\omega_F:=\{T\in \mathscr{T}_h\,;\  F  \subset \partial T \}$ is the element pair for a face $F \in \mathscr{F}_h$, and $\lambda_i^F$ are the barycentric coordinate functions associated with the vertices of the face $F \in \mathscr{F}_h$, which can be naturally extended to $\omega_F$. To extend the face residuals defined on $F$ to $\omega_F$, we introduce the extension operator as follows.
We first define the operator $\hat{E}:C(\hat{F})\rightarrow C(\hat{T})$
on the reference element $\hat{T}$ in $\R^3$ by
\begin{align*}
\hat{E}[\hat{p}](\hat{x},\hat{y},\hat{z}):=\hat{p}(\hat{x},\hat{y})\,,
\end{align*}
where $\hat{F}$ is the face of $\hat{T}$ lying on the $(\h{x},\h{y})$-plane.
By using the affine mapping $F_T(\h{\x})=A_T\hat{\x}+\a_T: \hat{T} \to T \in \omega_F$,
the general extension operator $E:C(F) \rightarrow C(\omega_F)$ can be introduced by
\begin{align} \label{aux_eq7}
E[p]|_T = \hat{E}[p\circ F_T] \circ F_T^{-1}\,,\ T \in \omega_F\,,
\end{align}
where the mappings $F_T$, $T\in \omega_F$, are chosen such that $\hat{F}$ is mapped to $F$ and $E[p]$ is well-defined on $\omega_F$ and continuous. The next lemma summarizes the important properties of the bubble functions, which can be easily verified by 
 the equivalence of norms in a finite-dimensional linear space and the standard scaling argument. 

\begin{lem}\label{bubble}
Let $k$ be a positive integer and $s$ be a positive real number. For any $T \in \mathscr{T}_h$ and $F \in \mathscr{F}_h$, there holds 
\begin{align} \label{auxeq_bubble1}
    \norm{\phi}_{0,T}\lesssim \norm{b_T^s \phi}_{0,T} \le \norm{\phi}_{0,T}\,,\quad \norm{\varphi}_{0.F}\lesssim \norm{b_F^s \varphi}_{0.F} \le \norm{\varphi}_{0,F}\,,
\end{align}
for all $\phi \in P_k(T)$ and $\varphi \in P_k(F)$, and 
\begin{align} \label{auxeq_bubble2}
    h_F^{1/2} \norm{\varphi}_{0,F}\lesssim \norm{b_F^sE(\varphi)}_{0,T}\lesssim h_F^{1/2} \norm{\varphi}_{0,F}\,,
\end{align}
for $T \in w_F$ and $\varphi \in P_k(F)$. 
\end{lem}

We are now in a position to state and prove the main result of this section.
\begin{thm} \label{efficiency}
Let the triplets $(\y^*,\p^*,\u^*)$ and $(\y_h^*,\p_h^*,\u_h^*)$ be the solutions to the continuous and discrete optimality systems \eqref{kkt1:1}-\eqref{kkt1:3} and
    \eqref{kkt2:1}-\eqref{kkt2:3}, respectively,
    and the multipliers $\llambda^*$ and $\llambda_h^*$ be given by \eqref{eq:confrederi} and \eqref{eq:disfrederi}. Then we have the efficiency estimate:
%
\begin{align}
\h{\eta}_h \lesssim \norm{\y_h^*-\y ^*}_{\curl,\Om}+\norm{\p_h^* -\p^*}_{\curl,\Om}&
+\norm{\u_h^*-\u^*}_{0,\Om} + \norm{\llambda^*_h-\llambda^*}_{0,\Om} \notag \\
&+ \osc_h(\y^d) + \osc_h(\f) + \osc_h(\u^d)\,.
\end{align}
\end{thm}

Before we start our proof, we remark that it is necessary to consider the error of the multiplier
$\norm{\llambda^*_h-\llambda^*}_{0,\Om}$ here  in order to estimate $\h{\eta}_h$, in comparison with  \cite{hoppe2015adaptive}, 
since the additional error estimator $\eta_{p,T}^{\sss (3)}$ is included in $\h{\eta}_h$. We start with the following local efficiency estimate:
\begin{align*}
    \eta_{p,T}^{(3)} & = \norm{\p^*_h - \p^* + \p^* - \mathbb{P}_h \p^*_h}_{0,T}\\ &\le    \norm{\llambda^*-\llambda_h^* - \alpha(\u^* - \u^*_h -\u^d  + \u_h^d)}_{0,T} + \norm{\p^*_h - \p^*}_{0,T}  \\
&     \leq \norm{\p^*-\p_h^*}_{0,T}+\norm{\llambda^*-\llambda_h^*}_{0,T}+\alpha \left(\norm{\u^*-\u_h^*}_{0,T} + \osc_T(\u^d)\right),
\end{align*}
by the definitions of $\llambda^*$, $\llambda_h^*$ and $\eta_{p,T}^{(3)}$ and the \mb{triangle inequality}. Our proof proceeds by establishing more local efficiency estimates for $\eta_T$, which are divided into the following four groups of estimates, for $T \in \mathscr{T}_h$ and $F \in \mathscr{F}_h$,
\begin{align*}
   & \left\{
        \begin{aligned}
           & \eta_{y,T}^{(1)} \lesssim h_T\norm{\u^*-\u_h^*}_{0,T}+\norm{\y^* -\y_h^*}_{\curl,T}+\osc_T(\f), \\
            &\eta_{p,T}^{(1)} \lesssim \norm{\p^*-\p_h^*}_{\curl,T}+\norm{\y^* -\y_h^*}_{\curl,T}+\osc_T(\y^d),
        \end{aligned}
    \right.\\
   & \left\{
        \begin{aligned}
           & \eta_{y,T}^{(2)} \lesssim \norm{\u^*-\u_h^*}_{0,T}+\norm{\y^* -\y_h^*}_{0,T} + \osc_T(\f), \\
        &    \eta_{p,T}^{(2)} \lesssim \norm{\p^*-\p_h^*}_{0,T},
        \end{aligned}
        \right.\\
        &\left\{
            \begin{aligned}
              &  \eta_{y,F}^{(1)} \lesssim h_T\norm{\u^*-\u_h^*}_{0,w_F}+\norm{\y^* -\y_h^*}_{\curl,w_F}+\eta_{y,T^+}^{(1)}+\eta_{y,T^-}^{(1)},\\
               & \eta_{p,F}^{(1)} \lesssim \norm{\p^*-\p_h^*}_{\curl,w_F}+\norm{\y^* -\y_h^*}_{\curl,w_F}+ \eta_{p,T^+}^{(1)}+\eta_{p,T^-}^{(1)},
            \end{aligned}
            \right.\\
           & \left\{
                \begin{aligned}
                &    \eta_{y,F}^{(2)} \lesssim \norm{\u^*-\u_h^*}_{0,w_F}+\norm{\y^* -\y_h^*}_{0,w_F}+ \eta_{y,T^+}^{(2)}+\eta_{y,T^-}^{(2)} + \osc_{T^+}(\f),\\
                &    \eta_{p,F}^{(2)} \lesssim \norm{\p^*-\p^*_h}_{0,w_F} + \eta_{p,T^+}^{(2)}+\eta_{p,T^-}^{(2)},
                \end{aligned}
                \right.
            \end{align*}
where $T^+$ and $T^-$ are two elements in $\omega_F$ with $F = T^+ \cap T^-$.


\begin{proof}
We give the proof of the above four groups of inequalities by the following four steps.
\begin{enumerate}[(1)]
\item  We start with $\eta_{y,T}^{\sss (1)}$ and readily see \mb{by the triangle inequality} that 
\begin{equation} \label{aux_eq1}
\eta_{y,T}^{(1)} \leq h_T\norm{\f_h+\u^*_h-\curl \mu^{-1} \curl\y_h^*-\sigma \y_h^* }_{0,T}+\osc_T(\f)\,.
\end{equation}
It is clear that $b_T$ vanishes on $\partial T$, and hence we can define $\z_h:=b_T(\f_h+\u_h^*-\curl \mu^{-1} \curl \y_h^*-\sigma\y_h^*) \in \textit{\textbf{H}}_0(\curl,\Om)$. Using the estimate \eqref{auxeq_bubble1} 
with $\phi = \f_h+\u_h^*-\curl \mu^{-1} \curl \y_h^*-\sigma\y_h^*$ and $s = 1/2$, we obtain 
\begin{align} \label{aux_eq2}
\norm{\f_h+\u_h^*-\curl \mu^{-1} \curl \y_h^*-\sigma\y_h^*}_{0,T}^2&\approx (\f_h+\u_h^*-\curl \mu^{-1} \curl \y_h^*-\sigma\y_h^*,\z_h)_{0,T} \notag \\
&=(\f_h-\f,\z_h)_{0,T}+ (\u^*_h - \u^*,\z_h)_{0,T}+B(\y^*-\y_h^*,\z_h)\,.
\end{align}
By estimates \eqref{aux_eq1} and \eqref{aux_eq2}, and the inverse inequality:
\begin{equation*}
\norm{\z_h}_{\curl,T}\lesssim h_T^{-1}\norm{\z_h}_{0,T}\,,
\end{equation*}
as well as the norm equivalence: $\norm{\z_h}_{0,T} \approx \norm{\f_h+\u_h^*-\curl \mu^{-1} \curl \y_h^*-\sigma\y_h^*}_{0,T}$, we can derive  
\begin{equation*}
\eta_{y,T}^{(1)} \lesssim h_T\norm{\u^*-\u_h^*}_{0,T}+\norm{\y^* -\y_h^*}_{\curl,T}+\osc_T(\f)\,.
\end{equation*}
The estimate of $\eta_{p,T}^{\sss (1)}$ is similar. We  note 
 \begin{equation} \label{aux_eq5}
    \eta_{p.T}^{(1)} \lesssim h_T\norm{\curl\y_h^d+\curl\mu^{-1} \curl \p_h^*+\sigma \p_h^*}_{0,T}+\osc_T(\y^d)\,,
\end{equation}
and define $\z_h := b_T(\curl\y^d_h+\curl\mu^{-1} \curl \p_h^*+\sigma \p^*_h)\in \textit{\textbf{H}}_0(\curl,\Omega)$. Then a similar estimate as above gives
\begin{align}
 \norm{\curl\y_h^d+ \curl\mu^{-1} \curl \p_h^*+\sigma \p_h^*}^2_{0,T} \lesssim &(\curl\y_h^d+\curl\mu^{-1} \curl\p_h^* +\sigma\p_h^*,\z_h)_{0,T} \notag \\
 \lesssim & (\curl\y^d  +\curl\mu^{-1} \curl\p_h^* +\sigma\p_h^*,\z_h)_{0,T} - (\curl \y^*, \curl\z_h )_{0,T}  \notag  \\
 &  + (\curl \y_h^d - \curl \y^d, \z_h)_{0,T} + (\curl(\y^* - \y^*_h), \curl z_h)_{0,T} \label{aux_eq3} \\
 \lesssim &\norm{\curl \y_h^d -\curl \y^d}_{0,T}\norm{\z_h}_{0,T} + \norm{\y_h^*-\y^*}_{\curl,T} \norm{\z_h}_{\curl,T} \notag \\
  & +\norm{\p^*_h-\p^*}_{\curl,T} \norm{\z_h}_{\curl,T} \label{aux_eq4}\,, 
\end{align}
\mb{where in \eqref{aux_eq3} we have used $(\curl \y^*_h, \curl \z_h)_{0,T} = 0$ from the fact that $\y^*_h$ is a first-order polynomial on $T$, and in \eqref{aux_eq4} we have used 
\begin{align*}
    & (\curl\y^d  +\curl\mu^{-1} \curl\p_h^* +\sigma\p_h^*,\z_h)_{0,T} - (\curl \y^*, \curl\z_h )_{0,T} \\ = & B(\p_h^* - \p^*, \z_h) \lesssim \norm{\p^*_h-\p^*}_{\curl,T} \norm{\z_h}_{\curl,T}\,.
\end{align*}
Further applying the inverse estimate for $\norm{\z_h}_{\curl,T}$ in \eqref{aux_eq4} and recalling \eqref{aux_eq5}, 
we come to}
\begin{equation*}
\eta_{p,T}^{(1)} \lesssim \norm{\p^*-\p_h^*}_{\curl,T}+\norm{\y^* -\y_h^*}_{\curl,T} + \osc_T(\y^d)\,.
\end{equation*}
\item Define $\z_h := \div(\f_h-\sigma\y_h^*)b_T^2$ with  $\norm{\z_h}_{0,T} \approx \norm{\div(\f_h-\sigma\y_h^*)}_{0,T}$. 
It is clear that $\nabla \z_h$ is a polynomial on $T$ and vanishes on the boundary $\partial T$, which gives 
$\nabla \z_h \in \textit{\textbf{H}}_0(\curl,\Omega)$. By a direct calculation, we have 
\begin{align}
    (\div(\f_h-\sigma\y_h^*), \z_h)_{0,T} &= (\div(\f_h-\sigma \y_h^* +\u_h^* ),\z_h)_{0,T} \notag \\
    & = (\div(\f-\sigma\y_h^*+\u_h^*),\z_h)_{0,T} + (\div(\f_h-\f),\z_h)_{0,T}  \notag \\
    & = (-\sigma \y^*+\sigma \y_h^*+\u^* -\u^*_h, \nabla \z_h)_{0,T} + (\div(\f_h-\f),\z_h)_{0,T}\,, \label{aux_eq6}
\end{align}
\mb{where we have used the following observation in the last equality: 
\begin{equation*}
    B(\y^*,\nabla \z_h) = (\sigma\y^*, \nabla \z_h)_{0,\Omega} = (\f+\u^*, \nabla \z_h)_{0,\Omega}\,,
\end{equation*}
which is from (\ref{kkt1:1}) with the test function $\phii = \nabla \z_h$.} Again, by Lemma \ref{bubble} and the inverse estimate, we can derive from the definition of $\eta_{y,T}^{\sss (2)}$ and \eqref{aux_eq6} that
\begin{equation*}
\eta_{y,T}^{(2)}\lesssim \norm{\u_h^*-\u^*}_{0,T}+\norm{\y^*_h-\y^*}_{0,T}+\osc_T(\f)\,.
\end{equation*}
Likewise, for $\eta_{p,T}^{\sss (2)}$, taking $\z_h = \div(\sigma\p_h^*)b_T^2$ and observing from (\ref{kkt1:2}):
\begin{equation*}
(\sigma\p^*, \nabla \z_h)_{0,\Omega} = (\curl\y_*-\y^d, \curl \nabla \z_h)_{0,\Omega} = 0\,,
\end{equation*}
we can derive, by almost the same arguments as above, that
\begin{equation*}
\eta_{p,T}^{(2)} \lesssim \norm{\p^*-\p_h^*}_{0,T}\,.
\end{equation*}
\item Since the lowest-order edge element is used for the discretization, $\curl\y_h^*$ is a piecewise constant vector. Then $\gamma_t(\mu^{-1} \curl \y_h^*)$ is a polynomial defined on $F$ and can be extended to $w_F$ by the extension operator $E$ introduced in \eqref{aux_eq7}. 
Define $$\z_h:=b_FE([\gamma_t(\mu^{-1} \curl \y_h^*)]_F) \in \textit{\textbf{H}}_0(\curl,\Om).$$ 
By the estimate \eqref{auxeq_bubble1} and integration by parts over $\omega_F$, we have 
\begin{align}
(\eta_{y,F}^{(1)})^2 =  h_F\norm{[\gamma_t(\mu^{-1} \curl \y_h^*)]_F}_{0,F}^2 \approx & h_F([\gamma_t(\mu^{-1} \curl \y_h^*)]_F,\z_h)_{0,F} \notag\\
= & h_F(\curl \mu^{-1} \curl \y_h^*+\sigma \y_h^*-\f-\u_h^*, \z_h)_{0.w_F} \notag \\
&-h_F(\mu^{-1} \curl \y_h^*-\mu^{-1} \curl \y^*, \curl \z_h)_{0,w_F} \notag \\
&+h_F(\u_h^*-\u^*,\z_h)_{0,w_F}+ h_F(\sigma \y^*-\sigma \y^*_h,\z_h)_{0,w_F}\,. \label{eff3}
\end{align}
The estimate \eqref{auxeq_bubble2} and the inverse estimate give us
\begin{align}
h_F^{1/2}\norm{[\gamma_t(\mu^{-1} \curl \y_h^*)]_F}_{0,F}\approx \norm{\z_h}_{0,w_F}\,,
\quad 
\norm{\curl \z_h}_{0,w_F}\lesssim h_F^{-1}\norm{\z_h}_{0,w_F}\,. \label{e2:p2}
\end{align}
Combining (\ref{e2:p2}) with the formula (\ref{eff3}), we get
\begin{equation*}
\eta_{y,F}^{(1)}\lesssim \eta_{y,T^+}^{(1)}+\eta_{y,T^-}^{(1)}+\norm{\y_h^*-\y^*}_{\curl,w_F}+h_F\norm{\u_h^*-\u^*}_{0,w_F}\,.
\end{equation*}
For $\eta_{p,F}^{\sss (1)}$, let $\z_h := b_F E([\gamma_t(-\mu^{-1} \curl \p_h^*+\curl\y_h^*)]_F)$. By similar calculations, it follows that
\begin{align}
(\eta_{p,F}^{(1)})^2 \approx & h_F([\gamma_t(-\mu^{-1} \curl \p_h^* +\curl \y_h^*)]_F,\z_h)_{0,F} \notag \\
= &h_F(-\curl\mu^{-1} \curl \p_h^*-\curl \y^d - \sigma \p^*_h, \z_h)_{0,w_F} \notag \\
&- h_F(-\mu^{-1} \curl \p^*_h + \curl \y_h^* - \y^d, \curl \z_h)_{0,w_F} + h_F(\sigma \p_h^*, \z_h)_{0,w_F} \label{aux_eq8} \\
 = &  h_F(-\curl\mu^{-1} \curl \p_h^*-\curl \y^d - \sigma \p^*_h, \z_h)_{0,w_F}  + h_F B(\p_h^*-\p^*,\z_h) \notag \\
 & + h_F (\curl(\y^* - \y_h^*), \curl \z_h)_{0,\omega_F} \,,
\end{align}
\mb{where we have used $(\curl \curl \y_h^*, \z_h)_{0, T} = 0$ for $T \in \omega_F$ in \eqref{aux_eq8}.} 
Then, by the inverse estimate and Lemma \ref{bubble},  a direct estimate leads to 
\begin{equation*}
\eta_{p,F}^{(1)} \lesssim \eta_{p,T^+}^{(1)}+\eta_{p,T^-}^{(1)}+\norm{\p^*-\p^*_h}_{0,w_F} + \norm{\y_h^*-\y^*}_{\curl,w_F}\,.
\end{equation*}
\item For $\eta_{y,F}^{(2)}$, we define $\z_h := b_F^2 E([\gamma_n(\f_h+\u_h^*-\sigma\y_h^*)]_F)$. It is easy to see that $\nabla \z_h \in \textit{\textbf{H}}_0(\curl,\Om)$. Taking $\phii = \nabla \z_h$ in (\ref{kkt1:1}) gives 
\begin{equation}\label{peff4:1}
(\f+\u^*-\sigma\y^*, \nabla \z_h)_{0,w_F} = 0\,.
\end{equation}
\mb{We note by the triangle inequality that 
\begin{equation} \label{eq:prfour1}
   \eta_{y,F}^{(2)}\le h_F^{1/2}\norm{[\gamma_n(\f_h+\u_h^*-\sigma\y_h^*)]_F}_{0,F} + h_F^{1/2} \norm{[\gamma_n(\f  - \f_h)]_F}_{0,F}  \,.
\end{equation}
Then we have, again by the property of the bubble functions \eqref{auxeq_bubble1} and integration by parts over $\omega_F$,}
\begin{align}
   \norm{[\gamma_n(\f_h+\u_h^*-\sigma\y_h^*)]_F}^2_{0,F} \approx& ([\gamma_n(\f_h+\u_h^*-\sigma\y_h^*)]_F,\z_h)_{0,F} \notag\\ 
    \lesssim  &(\div(\f-\sigma\y_h^*),\z_h)_{0,w_F}+ ([\gamma_n(\f_h-\f)]_F,\z_h)_{0,F} \notag
    \\&+ \mb{(\f + \u_h^* - \sigma \y_h^*, \nabla \z_h)_{0,\omega_F}} \,, \label{eq:prfour2}
\end{align}
\mb{where we can use \eqref{peff4:1} to rewrite the last term as 
\begin{align} \label{aux_eq9}
    (\f + \u_h^* - \sigma \y_h^*, \nabla \z_h)_{0,\omega_F}=   (\u_h^*-\u^*,\nabla \z_h)_{0,w_F}+(\sigma\y^*-\sigma \y_h^*,\nabla \z_h)_{0,w_F}.
\end{align}
Therefore, we obtain from \eqref{eq:prfour1}--\eqref{aux_eq9}, the Cauchy's inequality, 
applying the inverse estimate to $\norm{\nabla \z_h}_{0,\omega_F}$
and the trace inequality to $\norm{\z_h}_{0, F}$ (cf.\,\eqref{auxeq_trace}) that}
\begin{equation*}
\eta_{y,F}^{(2)} \lesssim \norm{\u^*-\u_h^*}_{0,w_F}+\norm{\y^* -\y_h^*}_{0,w_F}+ \eta_{y,T^+}^{(2)}+\eta_{y,T^-}^{(2)} + \osc_{T^+}(\f)\,,
\end{equation*}
\mb{where we have used the trivial bound that $h_F^{1/2} \norm{[\gamma_n(\f  - \f_h)]_F}_{0,F}  \le  \osc_{T^+}(\f)$.} 
The estimate for $\eta^{\sss (2)}_{p,F}$ follows from the 
same (even simpler) argument. In fact, we can define $\z_h = b_F^2 E([\gamma_n(\sigma\p_h^*)]_F)$, which implies $\nabla \z_h \in \textit{\textbf{H}}_0(\curl,\Om)$ and, by \eqref{kkt1:2}, $(\sigma \p^*, \nabla \z_h)_{0,\Omega} = 0$. Then, a typical calculation gives  
\begin{align*}
    \norm{[\gamma_n(\sigma\p_h^*)]_F}^2_{0,F} \approx& ([\gamma_n(\sigma\p_h^*)]_F,\z_h)_{0,F} =  (\div(\sigma\p_h^*),\z_h)_{0,w_F} + (\sigma\p_h^* - \sigma \p^*, \nabla \z_h)_{0,w_F}\,,
 \end{align*}
which yields, by Lemma \ref{bubble} and the inverse estimate, 
\begin{align*}
    \eta_{p,F}^{(2)} \lesssim \norm{\p^*-\p^*_h}_{0,w_F} + \eta_{p,T^+}^{(2)}+\eta_{p,T^-}^{(2)}\,.
\end{align*}
\end{enumerate}
Theorem\,\ref{efficiency} follows now by adding up the above local 
efficiency estimates over all $T \in \T{h}$.
\end{proof}

\section{Convergence}\label{converge}
We devote this whole section to establish our main result that the sequence of  
adaptively generated finite element
solutions $\{(\yk,\pk,\uk)\}_{k \ge 0}$ converges strongly to the true solution $(\y^*,\p^*,\u^*)$.
\begin{thm}\label{thm:converge}
Let $\{(\yk,\pk,\uk)\}_{k \ge 0}$ be the sequence of  discrete triplets generated by the adaptive Algorithm \ref{alg:Framwork} and $(\y^*,\p^*,\u^*)$ be the solution to the system \eqref{kkt1:1}-\eqref{kkt1:3}. Then we have
the strong convergences:
\begin{equation}
\lim\limits_{k \to \infty} \norm{\y^*_k-\y^*}_{\curl,\Omega} = 0\,,\ \lim\limits_{k \to \infty} \norm{\pk-\p^*}_{\curl,\Omega} = 0 \ \text{and}\  \lim\limits_{k \to \infty} \norm{\uk-\u^*}_{0,\Omega} = 0\,.
\end{equation}
\end{thm}


As already pointed out in the introduction, the first step of the proof of convergence is the introduction of a limiting minimization problem that characterizes the limit of the discrete solutions to the system \eqref{kkt2:1}-\eqref{kkt2:3} while the second step is to show that the solution to the limiting problem actually coincides with the one to \eqref{kkt1:1}-\eqref{kkt1:3}. To do so,
let us first state some helpful auxiliary results on the convergence behavior of the adaptive meshes
 $\{\mathscr{T}_k\}_{k \ge 0}$.
We associate each triangulation $\mathscr{T}_k$ with a mesh-size function $h_k \in  L^\infty(\Omega)$,  defined by $h_{k}|_T = h_T$ for $T\in \mathscr{T}_k$. Thanks to the monotonicity of $\{h_k\}_{k \ge 0}$, we are allowed to define a limiting mesh-size function by the pointwise limit: $h_k(\x) \to h_\infty(\x)$, as $k \rightarrow \infty$.  
\mb{It is also known that} the pointwise convergence of $\{h_k\}$ can be \mb{further improved} to the uniform convergence \cite[Lemma 4.2]{nochetto2009theory}:
\begin{equation} \label{uniform}
\lim\limits_{k \rightarrow \infty} \norm{h_k-h_\infty}_{\infty,\Omega} = 0\,.
\end{equation}
We should note that $h_\infty(\x) \not\equiv 0$ in general. If $h_\infty (\x) > 0$ at some point $\x$, there is an element $T$ containing $\x$ and an index $k(\x)$ depending on $\x$ such that $T \in \mathscr{T}_l$ for all $l\geq k(\x)$. This observation motivates us to split $\mathscr{T}_k$ into two classes of elements:
\begin{equation} \label{eq:splitting}
\mathscr{T}_k^+ := \bigcap\limits_{l \geq k} \mathscr{T}_l \quad \text{and}\quad
\mathscr{T}_k^0 :=\mathscr{T}_k \backslash {\mathscr{T}_k}^+\,.
\end{equation}
$\mathscr{T}_k^+$ consists of the elements in $\mathscr{T}_k$ that are not refined after the $k$th iteration, while $\mathscr{T}_k^0\subset \mathscr{T}_k$ contains the elements that are refined at least once in the subsequent iterations.
We are thus allowed to 
 decompose the domain $\Omega$ into two parts: $\Omega_k^+ := \bigcup_{\sss T \in \mathscr{T}_k^+}T$ and $\Omega_k^0 := \bigcup_{\sss T \in \mathscr{T}_k^0}T$. A direct application of (\ref{uniform}) and the uniform shape regularity of $\{\mathscr{T}_k\}$ yields the following result (cf.\,\cite[Corollary 7.1]{nochetto2009theory}):
\begin{equation}\label{uniformstrong} 
\lim\limits_{k \rightarrow \infty} \norm{h_k}_{\infty,\widetilde{\Omega}_k^0} = 0 \,,
\end{equation}
where $\widetilde{\Omega}_k^0:=\bigcup_{\sss T\in  \mathscr{T}_k^0} \widetilde{\Omega}_T$ is the extended neighborhood of $\Omega_k^0$. 
We remark that this uniform convergence result for the mesh-size functions is crucial for our subsequent analysis. We next give 
an interesting characterization of the limiting behavior of $L^2$-projections $\{\mathbb{P}_k\}_{k \ge 0}:= \{\mathbb{P}_{\U_{k}}\}_{k \ge 0}$ with the help of the convergence property of the adaptive meshes $\{\mathscr{T}_k\}_{k \ge 0}$, which establishes a connection between the limit of mesh-size functions, the $L^2$-projections and the limiting problem that we shall propose and deal with in the next section.
\begin{prop}\label{average}
Let $\{\mathbb{P}_k\}_{k \ge 0}$ be the orthogonal $L^2$-projections (defined by \eqref{averoper}) associated with the adaptive meshes $\{\mathscr{T}_k\}_{k \ge 0}$ generated by Algorithm \ref{alg:Framwork}. Then for each $\f \in \textit{\textbf{L}}^2(\Omega)$, the limit of the sequence $\{\mathbb{P}_k \f\}_{k \ge 0}$, denoted by $\mathbb{P}_\infty \f\,,$ exists as $k \rightarrow \infty$. Furthermore, the corresponding limiting operator $\mathbb{P}_\infty$ is also an orthogonal $L^2$-projection with the range and kernel given by
\begin{equation} \label{eq:ranandker}
    {\rm ran}(\mathbb{P}_\infty) = \U_\infty := \overline{\bigcup_{k \ge 0}\U_k}^{L^2} \quad  \text{and} \quad \ker(\mathbb{P}_\infty) = \U_\infty^\bot\,.
\end{equation}
\end{prop}
\begin{proof}
It suffices to show that the limit of the sequence $\{ \mathbb{P}_k \f\}$ exists for all $\f \in \textit{\textbf{C}}_c^\infty(\Omega)$ (infinitely differentiable vector-valued functions with compact support in $\Om$), by the facts that the operators
$\{\mathbb{P}_k\}$ are uniformly bounded and the space $\textit{\textbf{C}}_c^\infty(\Omega)$ is dense in $\textit{\textbf{L}}^2(\Omega)$. Given two iteration indices $k_1, k_2$ with $k_2 >  k_1$, define $\mathscr{T}_{k_1, k_2}^0 := \mathscr{T}_{k1} \backslash (\mathscr{T}_{k1}\cap \mathscr{T}_{k2})$, which is a subset of  
$\mathscr{T}_{k1}^0$ consisting of the elements that are refined between the $k_1$th iteration and $k_2$th iteration.
We then have, by the equivalent definition of $\mathbb{P}_k$ in \eqref{averoper}, 
\begin{align*}
\norm{\mathbb{P}_{k_1}\f - \mathbb{P}_{k_2}\f}_{0,\Omega}^2 &=\Big \| \sum\limits_{T\in \mathscr{T}_{k_1, k_2}^0}\sum\limits_{T_i \subset T,T_i \in \mathscr{T}_{k_2}}\left(\frac{1}{|T|} \int_T \f(\x) d\x - \frac{1}{|T_i|}\int_{T_i}\f(\x) d\x \right)\chi_{T_i} \Big \|_{0,\Omega}^2 \\
&=\sum\limits_{T\in \mathscr{T}_{k1k2}^0}\sum\limits_{T_i \subset T,T_i \in \mathscr{T}_{k_2}} \left(\frac{1}{|T|} \int_T \f(\x) d\x - \frac{1}{|T_i|}\int_{T_i}\f(\x) d\x \right)^2 |T_i|\,,
\end{align*}
where $\chi_{T_i}$ is the characteristic function of $T_i$.  Recalling the limiting behavior of  mesh-size functions $\{h_k\}$ in  \eqref{uniformstrong}, we have that for any $\delta > 0$, there exists an index $k(\delta)$ depending on $\delta$ such that for all $k > k(\delta)$, there holds $\norm{h_k}_{\infty,\widetilde{\Omega}_k^0}\leq \delta$. 
Combining it with the uniform continuity of $\f$, we have that for any $\varepsilon>0$, there exists an index $k(\varepsilon, \delta)$ depending on $\varepsilon$ and $\delta$ such that for any integers $k_1,k_2$ satisfying $k_2>k_1>k(\varepsilon, \delta)$, 
and for any elements $T \in \mathscr{T}_{k_1, k_2}^0$, $T_i\in \mathscr{T}_{k_2}$ with $T_i \subset T$, it holds that
\begin{equation*}
\Big |\frac{1}{|T|} \int_T \f(\x) d\x - \frac{1}{|T_i|}\int_{T_i}\f(\x) d\x \Big | \leq \varepsilon\,.
\end{equation*}
We hence have that $\{\mathbb{P}_k \f\}$ is a Cauchy sequence in $\textit{\textbf{L}}^2(\Omega)$. Then it follows from the completeness of $\textit{\textbf{L}}^2(\Omega)$ that the limit of $\{\mathbb{P}_k \f\}$ exists.
 We have proved that for any $\f \in \textit{\textbf{L}}^2(\Omega)$, $\mathbb{P}_\infty f$ is well-defined, which further allows us to define the bounded linear operator $\mathbb{P}_\infty$ on $\textit{\textbf{L}}^2(\Omega)$, by the  uniform boundedness principle. We now show that $\mathbb{P}_\infty$ is an orthogonal $L^2$-projection with the property \eqref{eq:ranandker}. To do so, we first observe that for any $\g\in \textit{\textbf{L}}^2(\Omega)$, it holds that
$$(\mathbb{P}_\infty\f,\g)_{0,\Omega} = \lim\limits_{n\to \infty}(\mathbb{P}_k\f,\g)_{0,\Omega} = \lim\limits_{n\to \infty}(\mathbb{P}_k\g,\f)_{0,\Omega} = (\f,\mathbb{P}_\infty\g)_{0,\Omega}\,,$$ 
which implies that $\mathbb{P}_\infty$ is self-adjoint. On the other hand, we have 
\begin{align*}
\norm{(\mathbb{P}_\infty^2-\mathbb{P}_\infty)\f}_{0,\Omega} = &\norm{(\mathbb{P}_\infty^2-\mathbb{P}_k \mathbb{P}_\infty + \mathbb{P}_k \mathbb{P}_\infty - \mathbb{P}_k^2 + \mathbb{P}_k -\mathbb{P}_\infty)\f}_{0,\Omega} \\
\leq &\norm{(\mathbb{P}_\infty-\mathbb{P}_k)(\mathbb{P}_\infty\f)}_{0,\Omega} + \norm{\mathbb{P}_k} \norm{(\mathbb{P}_\infty-\mathbb{P}_k)\f}_{0,\Omega} \\
&+ \norm{(\mathbb{P}_k-\mathbb{P}_\infty)\f}_{0,\Omega} \to \text{0} \ \ \text{as}\  k \to \text{0}\,,
\end{align*}
that is, $\mathbb{P}_\infty^2 = \mathbb{P}_\infty$. Thus, we can conclude  that $\mathbb{P}_\infty$ is an orthogonal $L^2$-projection. To characterize its range and kernel,
we denote by 
$\mathbb{P}$ the orthogonal projection associated with the closed space $\U_\infty$ defined in \eqref{eq:ranandker}. By definition, there holds ${\rm ran}(\mathbb{P})$ = $\U_\infty$ and $\ker(\mathbb{P})$ = $\U_\infty^\bot$.
We readily see from the definition of $\mathbb{P}_\infty$ that ${\rm ran}(\mathbb{P}_\infty) \subset \U_\infty$, since every element in ${\rm ran}(\mathbb{P}_\infty)$ can be approximated by a sequence from $\bigcup_{k \ge 0}\U_k$. Conversely, to prove  $\U_\infty \subset {\rm ran}(\mathbb{P}_\infty)$, we note that for any $ \f \in \U_n$, $\mathbb{P}_k\f = \f$ holds for each $k \ge n$. Then, letting $k \to \infty$ gives us 
$$\mathbb{P}_\infty\f = \lim\limits_{n\to\infty} \mathbb{P}_k\f = \f\,,$$ which indicates $\U_n \subset {\rm ran}(\mathbb{P}_\infty)$.  Since $n$ is arbitrary, we readily see $\overline{\bigcup_{k \ge 0}\U_k}^{L^2} \subset {\rm ran}(\mathbb{P}_\infty)$. The proof is complete.
\end{proof}

\subsection{The limiting problem} \label{sec:limit}
In order to find the limit point of the discrete triplets $\{(\yk,\pk,\uk)\}_{k \ge 0}$, we first define several limiting spaces as the closure of the union of discrete spaces at each level:
\begin{equation}\label{deflim}
\V_\infty := \overline{\bigcup_{k \ge 0}\V_k}^{\textit{\textbf{H}}(\curl)}, \quad \U_\infty^{ad} := \overline{\bigcup_{k \ge 0}\U_k^{ad}}^{L^2}.
\end{equation}
which immediately implies $\U_\infty^{ad} = \U_\infty\cap \U^{ad}$, where the space $\U_\infty$ is defined in \eqref{eq:ranandker}. Since $\U_\infty^{ad}$ is a convex subset of $\textit{\textbf{L}}^2(\Omega)$, it is closed in the weak topology if and only if it is closed in the strong topology induced by the norm. We hence have the following key lemma.
\begin{lem}\label{weak}
Let $\u_k \in \U_k^{ad},\ k\ge 0$, be a sequence weakly converging to a $\u$ in $\textit{\textbf{L}}^2(\Omega)$. Then $\u \in \U_\infty^{ad}$ holds.
\end{lem}

We now consider the following limiting problem defined on the limiting spaces:
\begin{align}
\text{minimize}~~ &  J(\y_\infty ,\u_\infty)=\frac{1}{2}\norm{\curl \y_\infty -\y^d}^2_{0,\Omega}+\frac{\alpha}{2}\norm{\u_\infty-\u^d}^2_{0,\Omega} \label{lim:pro1}\\
\text{over}~~ &  (\y_\infty,\u_\infty) \in \V_\infty \times \U_\infty^{ad} \notag \\
\text{subject  to}~~ & B(\y_\infty , \phii_\infty) = (\f + \u_\infty ,\phii_\infty)_{0,\Omega} \quad \forall \ \phii_\infty \in \V_\infty\,. \label{lim:pro2}
\end{align}
The existence and uniqueness of the minimizer to the above problem can be obtained  by the standard arguments as the continuous and discrete cases, and the reduced cost functional for the limiting problem is defined by
$$J_\infty(\u_\infty) := \frac{1}{2}\norm{\curl \y_\infty(\u_\infty)-\y^d}_{0,\Om}^2+\frac{\alpha}{2}\norm{\u_\infty-\u^d}_{0,\Om}^2\,,$$
where $\y_\infty(\u_\infty)$ denotes the solution to the equation \eqref{lim:pro2} with the source $\u_\infty$. The rest of this subsection is devoted to proving that the limit point of $\{(\yk,\pk,\uk)\}_{k \ge 0}$ exists and happens to be the solution to the limiting optimization problem. As pointed out in the introduction, to compensate the lack of the discrete compactness, we shall investigate the weak limit of the sequence $\{\u^*_k\}_{k \ge 0}$ first, and then improve it to the strong convergence; see the next two theorems.


\begin{thm} \label{thm:limweak}
Suppose that $\{(\uk,\yk)\}_{k \ge 0}$ is the sequence of discrete optimal controls and states generated by
 Algorithm \ref{alg:Framwork}, and $(\u_\infty^*, \y_\infty^*)$ is the optimal control and state to the limiting optimization problem (\ref{lim:pro1})-(\ref{lim:pro2}). Then, we have the following weak convergences: as $k \to \infty$,
\begin{equation*}
\uk\weak \u_\infty^* \ \text{{\rm in}} \ \textit{\textbf{L}}^2(\Omega) \quad  \text{{\rm and}} \quad  \yk\weak \y_\infty^* \ \text{{\rm in}} \ \textit{\textbf{H}}_0(\curl,\Omega)\,. 
\end{equation*}
\end{thm}
\begin{proof}
Our proof starts with a simple but important observation that the sequence $\{\uk\}$ is bounded,
which is not a trivial fact due to the unboundedness of $\U^{ad}$. Indeed, for a fixed $\u \in \U^{ad}$, we have
\begin{align*}
\frac{\alpha}{2} \norm{\uk-\u^d}_{0,\Omega}^2 \leq J_k(\uk) \leq J_k(\mathbb{P}_k\u) &= \frac{1}{2}\norm{\curl \y_k(\mathbb{P}_k\u) -\y^d}_{0,\Omega}^2+\frac{\alpha}{2}\norm{\mathbb{P}_k\u-\u^d}^2_{0,\Omega} \notag\\
& \lesssim \norm{\u}_{0,\Omega}^2 + \norm{\y^d}_{0,\Omega}^2 + \norm{\u^d}_{0,\Omega}^2,
\end{align*}
which gives the boundedness of $\{\uk\}$ \mb{in $\textit{\textbf{L}}^2(\Omega)$}. Then the boundedness of $\{\yk\}$ follows immediately. Hence, by the  Banach-Alaoglu theorem, we can extract a subsequence $\{(\y^*_{k_n},\u^*_{k_n})\}_{n \ge 1}$ of $\{(\yk,\uk)\}_{k \ge 0}$ such that $\u^*_{k_n}\weak \w$ in $\textit{\textbf{L}}^2(\Omega)$ and $\y^*_{k_n}\weak \y$ in  $\textit{\textbf{H}}(\curl,\Omega)$, as $n$ tends to infinity, which further implies that $\{\y^*_{k_n}\}$ and $\{\curl \y^*_{k_n}\}$ weakly converge to $\y$ and $\curl \y$, respectively, in $\textit{\textbf{L}}^2(\Omega)$ \mb{(since $\textit{\textbf{H}}(\curl,\Omega)$ is continuously imbedded in $\textit{\textbf{L}}^2(\Omega)$ and $\curl$ is a bounded linear operator from $\textit{\textbf{H}}(\curl,\Omega)$ to $\textit{\textbf{L}}^2(\Omega)$)}.  
Moreover, Lemma \ref{weak} yields $\w \in \U_{\infty}^{ad}$. Noting $\V_l \subset \V_{k_n}$ for $l \le k_n$ and that there holds
\begin{equation} \label{aux_eq10}
B(\y^*_{k_n},\v_l) = (\f+\u_{k_n}^*,\v_l)_{0,\Omega} \quad \text{for} \ l \leq k_n \ \text{and} \ \v_l \in \V_l\,,
\end{equation}
we let $n$ tend to infinity in \eqref{aux_eq10} and obtain, by the weak convergences of $\{\u_{k_n}^*\}$ and $\{\y_{k_n}^*\}$,
\begin{equation*}
B(\y,\v_l)=(\f+ \w,\v_l)_{0,\Omega} \quad \text{for all} \ l\ge 0 \ \text{and} \  \v_l\in \V_l\,.
\end{equation*}
Since $\bigcup_{l \geq 0}\V_l$ is dense in  $\V_\infty$, we readily have 
\begin{equation} \label{aux_eq11}
B(\y,\v_\infty)=(\f+ \w,\v_\infty)_{0,\Omega}  \quad  \forall \ \v_\infty \in \V_\infty\,.
\end{equation}
which means that $(\y,\w)$ satisfies the constraint (\ref{lim:pro2}) of the limiting problem, that is, $\y = \y_\infty(\w)$.

We claim that $\w$ is the minimizer $\u_\infty^*$ of the cost functional $J_\infty$ over the set $\U_{\infty}^{ad}$, namely, 
\begin{equation} \label{eq:inpoclaim}
    J_\infty(\w) \le J_\infty(\u_\infty) \quad \forall\  \u_\infty\in \U^{ad}_\infty\,, 
\end{equation}
which, by \eqref{aux_eq11}, also implies that the weak limit $\y$ of the sequence $\{\y^*_{k_n}\}$ is $\y_\infty^*$. To prove the claim, we first note
\begin{align}
J_\infty(\w)&=\frac{1}{2}\norm{\curl \y_\infty(\w) -\y^d}_{0,\Omega}^2 + \frac{\alpha}{2}\norm{\w-\u^d}_{0,\Omega}^2 \notag\\
&\leq \liminf_{n \rightarrow \infty} \frac{1}{2}\norm{\curl \y_{k_n}^* -\y^d}^2_{0,\Omega} + \liminf_{n \rightarrow \infty}\frac{\alpha}{2}\norm{\u_{k_n}^*-\u^d}^2_{0,\Omega} = \liminf_{n \rightarrow \infty}J_{k_n}(\u^*_{k_n})\,, \label{eq:inte1}
\end{align}
by the uniform boundedness principle and the weak converges of $\{\curl\y^*_{k_n}\}$  and $\{\u^*_{k_n}\}$  to $\curl\y_\infty(\w)$ and $\w$ in $\textit{\textbf{L}}^2(\Omega)$. Clearly, by \eqref{eq:inte1} and the fact that $\u^*_{k_n}$ is the minimizer of $J_{k_n}$ over $\U^{ad}_{k_n}$, there holds, 
for any sequence, $\u_k\in \U_k^{ad}, k \ge 0$, 
\begin{equation}\label{lx0}
J_\infty(\w) \leq \liminf_{n \rightarrow \infty} J_{k_n}(\u^*_{k_n}) \leq
\limsup_{n \rightarrow \infty} J_{k_n}(\u^*_{k_n}) \leq \limsup_{k \rightarrow \infty} J_k(\u_k)\,.
\end{equation}
\mb{We next prove an auxiliary fact that for any $\u_\infty \in \U^{ad}_\infty$ and a sequence $\u_k\in \U_k^{ad}$ 
for $k \ge 0$}, if $\norm{\u_k - \u_\infty}_{0,\Omega} \to 0$ as $k \to \infty$, then 
\begin{equation} \label{lx3}
    \lim_{k \to \infty} \norm{\y_\infty(\u_\infty)-\y_k(\u_k)}_{\curl,\Omega} = 0\,,
\end{equation}
which readily gives 
\begin{equation} \label{lx4}
    \lim_{k \rightarrow \infty} J_k(\u_k) = J_\infty(\u_\infty)\,.
   \end{equation}
This fact  \eqref{lx4}, along with \eqref{lx0}, completes our proof of the claim \eqref{eq:inpoclaim}.  To prove \eqref{lx3}, we note, by the assumption,  
\begin{align*}
\limsup_{k \rightarrow \infty}\norm{\y_\infty(\u_\infty)-\y_k(\u_k)}_{\curl,\Omega}&\leq \limsup_{k \rightarrow \infty}(\norm{\y_\infty(\u_\infty)-\y_k(\u_\infty)}_{\curl,\Omega}+\norm{\y_k(\u_{\infty})-\y_k(\u_k)}_{\curl,\Omega}) \notag \\
&\lesssim \limsup_{k \rightarrow \infty}(\inf_{\v_k \in V_k} \norm{\y_\infty(\u_{\infty}) -\v_k}_{\curl,\Omega}+\norm{\u_{\infty} -\u_k}_{0,\Omega}) = 0\,,
\end{align*}
where we have used the Galerkin orthogonality and the density of $\bigcup_{k \ge 0}\V_k$ in $\V_\infty$.

\mb{By the above arguments, we can conclude} that for any subsequence $\{(\u^*_{k_n},\y^*_{k_n})\}_{n \ge 1}$ of $\{(\u_k^*,\yk)\}_{ k \ge 0}$, we can extract a subsequence weakly converging to $(\u_\infty^*, \y_\infty^*)$ in $\textit{\textbf{L}}^2(\Omega) \times \textit{\textbf{H}}_0(\curl,\Omega)$, which immediately yields the weak convergence of the whole sequence $\{(\u_k^*,\yk)\}_{ k \ge 0}$: 
 \begin{equation*}
    \uk\weak \u_\infty^* \ \text{{\rm in}} \ \textit{\textbf{L}}^2(\Omega)\ \  \text{{\rm and}} \ \    \yk\weak \y_\infty^* \ \text{{\rm in}} \ \textit{\textbf{H}}_0(\curl,\Omega)\,,\quad \text{as}\ k \to \infty.
\end{equation*}
The proof is complete.
\end{proof}
Thanks to the above results, we are ready to show the main theorem of this subsection.
\begin{thm}\label{limiting}
Under the same assumptions as in Theorem \ref{thm:limweak}, there holds
\begin{equation}\label{thm:uyk}
\lim\limits_{k \to \infty}\norm{\uk-\u_\infty^*}_{0,\Omega} = 0\    \text{and} \  \lim\limits_{k \to \infty}\norm{\yk-\y_\infty^*}_{\curl,\Omega} = 0\,.
\end{equation}
\end{thm}
\begin{proof}
We note from the claim \eqref{lx3} in the proof of Theorem \ref{thm:limweak} that the second convergence in \eqref{thm:uyk}
is a consequence of the first one. Hence, it suffices to prove the convergence of $\{\uk\}_{k \ge 0}$.
To this end, a direct calculation gives
\begin{align*}
 &\norm{\curl\yk-\curl\y^*_\infty}^2_{0,\Omega} + \alpha\norm{\u_k^*-\u_\infty^*}^2_{0,\Omega} \notag\\
=&\norm{\curl\yk-\y^d + \y^d  - \curl\y^*_\infty}^2_{0,\Omega} + \alpha\norm{\u_k^*-\u^d+\u^d-\u_\infty^*}^2_{0,\Omega} \notag \\
=&\norm{\curl\yk-\y^d}^2_{0,\Omega} + \norm{\curl\y_\infty^*-\y^d}^2_{0,\Omega} - 2(\curl\yk-\y^d,\curl\y_\infty^*-\y^d)_{0,\Omega}\notag\\
& + \alpha\norm{\uk-\u^d}^2_{0,\Omega} + \alpha\norm{\u_\infty^*-\u^d}^2_{0,\Omega} - 2\alpha(\uk-\u^d,\u_\infty^*-\u^d)_{0,\Omega}\,, \label{eq:inproeq}
\end{align*}
which, by taking the upper limit on both sides and using Theorem \ref{thm:limweak}, implies
\begin{equation} \label{eq:inter3}
    2J_\infty(\u_\infty^*) + \limsup_{k \rightarrow \infty} (\norm{\curl\yk-\curl\y^*_\infty}_{0,\Omega}^2+\alpha\norm{\u_k^*-\u_\infty^*}_{0,\Omega}^2) \leq \limsup_{k \rightarrow \infty} 2J_k(\uk).
\end{equation}
Then it follows that 
\begin{equation} \label{eq:inter2}
    J_\infty(\u_\infty^*) \le \limsup_{k \rightarrow \infty} J_k(\uk)\,.
\end{equation}
We choose a sequence $\u_k\in \U_k^{ad}, k \ge 0$, such that $\norm{\u_k - \u^*_\infty}_{0,\Omega} \to 0$, as $k \to \infty$, and then we can derive, by \eqref{eq:inter2}
and \eqref{lx4},  that
\begin{equation} \label{aux_eq15}
J_\infty(\u^*_\infty) \le \limsup_{k \to \infty} J_k(\u^*_k) \le \lim_{k \to \infty} J_k(\u_k) = J_\infty(\u^*_\infty)\,.
\end{equation}
Combining the above estimate \eqref{aux_eq15} with \eqref{eq:inter3}, we obtain the strong convergence of $\{\u_k^*\}_{k\ge 0}$.
\end{proof}
It is easy to write the optimality system for the limiting problem by a standard argument:
\begin{subequations}
  \begin{empheq}[left=\empheqlbrace]{alignat = 2}
&B(\y ^*_\infty ,\phii_\infty) = (\f+ \u^*_\infty,\phii_\infty)_{0,\Omega}  \quad  &\forall \ \phii_\infty \in \V_\infty\,, \label{kkt3:1}\\
&B(\p ^*_\infty,\varphii_\infty) = (\curl\y ^*_\infty-\y ^d,\curl\varphii_\infty)_{0,\Omega}   \quad  &\forall \ \varphii_\infty \in \V_\infty\,, \label{kkt3:2}\\
&(\p ^*_\infty+\alpha(\u^*_\infty-\u^d),\u_\infty- \u_\infty^*)_{0,\Omega} \geq 0 \quad  &\forall \ \u_\infty \in \U_\infty^{ad}\,,  \label{kkt3:3}
  \end{empheq}
\end{subequations}
and see that the sequence of  discrete triplets $\{(\yk,\pk,\uk)\}_{k \ge 0}$ converges strongly to the solution $(\y_{\infty}^*,\p_{\infty}^*,\u_{\infty}^*)$ to the limiting optimality system (\ref{kkt3:1})-(\ref{kkt3:3}). We remark that the variational inequality (\ref{kkt3:3}) will be used in the next subsection.

\subsection{Proof of convergence}
We have proved the strong convergence of the discrete solutions $\{(\yk,\pk,\uk)\}$ in Section \ref{sec:limit}, where we have only used the structure of the control problem, while the adaptive process does not have an essential involvement. \mb{We shall see that in the subsequent analysis, the error estimator $\h{\eta}_h$ defined in \eqref{eq:toterrind} and the marking requirement \eqref{eq:algocond} in Algorithm \ref{alg:Framwork} play a crucial role. 
To complete the proof of Theorem \ref{thm:converge}, we start with the following key lemma.}





\begin{lem} \label{lem:vaniestimator}
Let $\widetilde{T}_k\in \mathscr{T}_k$ be one of the elements that achieve the maximum value of $\h{\eta}_k(T)$ over $T \in \mathscr{T}_k$, i.e., $\h{\eta}_k(\widetilde{T}_k):= \max\limits_{T\in \mathscr{T}_k}\h{\eta}_k(T)$.
Then we have
\begin{equation}
\lim_{k\rightarrow \infty} \h{\eta}_k(\widetilde{T}_k)=0\,.
\end{equation}
\end{lem}


\begin{proof}
\mb{We start with the estimates for $\eta_{y,T}^{\sss (1)}$ and $\eta_{y,F}^{\sss (1)}$ on a given mesh $\T{k}$. A direct application of the inverse estimate and the triangle inequality gives 
\begin{align*}
    \eta_{y,T}^{(1)} \lesssim h_T\norm{\f}_{0,T} + h_T\norm{\u_h^*}_{0,T} + \norm{\curl \y_h^*}_{0,T} + h_T\norm{\y_h^*}_{0,T} \,,
\end{align*}
and, by the assumption of $\mu$ and the trace inequality \eqref{auxeq_trace}, we have 
\begin{align*}
    \eta_{y,F}^{(1)} &\lesssim h_F^{1/2} \norm{(\curl \y_h^*)|_-}_{0,F} + h_F^{1/2}\norm{(\curl \y_h^*)|_+}_{0,F}
     \lesssim  \norm{\curl \y_h^*}_{0,\omega_F}\,.
\end{align*}
We now consider $\eta_{y,T}^{\sss (2)}$ and $\eta_{y,F}^{\sss (2)}$. Similarly, we have 
\begin{align*}
    \eta_{y,T}^{(2)} \lesssim h_T \norm{\div \f}_{0,T} + \norm{\y_h^*}_{0,T}\,,
\end{align*}
and  
\begin{align} \label{aux_eq16}
    \eta_{y,F}^{(2)} \lesssim  h_F^{1/2} \norm{[\gamma_n(\f)]_F}_{0,F} + h_F^{1/2} \norm{[\u^*_h]_F}_{0,F} + h_F^{1/2} \norm{[\y^*_h]_F}_{0,F}\,
\end{align}
by the triangle inequality. Then the trace inequality \eqref{auxeq_trace} gives \
\begin{align} \label{aux_eq18}
    h_F^{1/2} \norm{[\u^*_h]_F}_{0,F} + h_F^{1/2} \norm{[\y^*_h]_F}_{0,F} \lesssim \norm{\u_h^*}_{0,\omega_F} + \norm{\y_h^*}_{\curl,\omega_F} \,,
\end{align}
since for $T \in \omega_F$, $\u_h^*|_T$ is a constant vector and $\y_h^*|_T \in \textit{\textbf{H}}^1(T)$ satisfies $\sqrt{2}|\nabla \y_h^*| = |\curl \y_h^*|$ which can be directly checked by using $\y_h^*|_T = \a_T \times \x + \b_T$ for some $\a_T, \b_T \in \R^3$. It is clear from \eqref{aux_assp} that 
$\norm{[\gamma_n(\f)]_F}_{0,F} > 0$, only if $F \subset \Gamma = \cup_{i = 1}^m \partial \Omega_i \backslash \partial \Omega$. Hence, it follows from \eqref{aux_eq16} and \eqref{aux_eq18} that 
\begin{align*}
    \eta_{y,F}^{(2)} \lesssim h_F^{1/2} \norm{[\gamma_n(\f)]_F}_{0,F \cap \Gamma} + \norm{\u_h^*}_{0,\omega_F} + \norm{\y_h^*}_{\curl,\omega_F}. 
\end{align*}
 The same analysis applies to the other terms in $\eta_k$. Then a careful but straightforward computation shows, for $T \in \T{k}$, 
\begin{align}   \label{mpost}
    \eta_k(T) \lesssim & \norm{\y^*_k}_{\curl,\omega_T} + \norm{\p^*_k}_{\curl,\omega_T} + \norm{\u^*_k}_{0,\omega_T} + h_T\norm{\curl \y^d}_{0,T} \notag \\
& + h_T \norm{\f}_{\div,T} + \sum_{F \in \partial T \cap \Gamma}  h_F^{1/2} \norm{[\gamma_n(\f)]_F}_{0,F}\,.                                                
\end{align}
Here $\omega_T$ denotes the union of elements that share a common face with $T$.} 
Then, it follows from \eqref{mpost} 
and the estimate
\begin{align*}
    \mb{h_{\widetilde{T}_k} \norm{\f}_{\div,\widetilde{T}_k} + \sum_{F \in \partial \widetilde{T}_k \cap \Gamma}  h_F^{1/2} \norm{[\gamma_n(\f)]_F}_{0,F} \lesssim  h^{1/2}_{\widetilde{T}_k} (\norm{\f}_{0,\Omega} + \norm{\div \f}_{0,\Omega} + \norm{[\gamma_n(\f)]_\Gamma}_{0,\Gamma})} 
\end{align*}
that
\begin{align}
\eta_k(\widetilde{T}_k)\lesssim & 
\mb{ h^{1/2}_{\widetilde{T}_k} {\rm C}(\f) } 
+\norm{\y_k^*-\y^*_\infty}_{\curl,\Omega}+\norm{\p_k^*-\p^*_\infty}_{\curl,\Omega}+\norm{\u_k^*-\u^*_\infty}_{0,\Omega} \notag \\& +h_{\widetilde{T}_k}\norm{\curl \y^d}_{0, \widetilde{T}_k}
+\norm{\y_\infty}_{\curl,\omega_{\widetilde{T}_k}}+\norm{\p_\infty}_{\curl,\omega_{\widetilde{T}_k}}+\norm{\u_\infty}_{0,\omega_{\widetilde{T}_k}}, \label{eq:interpr1}
\end{align}
\mb{where ${\rm C}(\f) := \norm{\f}_{0,\Omega} + \norm{\div \f}_{0,\Omega} + \norm{[\gamma_n(\f)]_\Gamma}_{0,\Gamma}$ is well-defined by \eqref{aux_assp}}. Hence, by the definition of $\h{\eta}_h$, we have the estimate for $\h{\eta}_k (\widetilde{T}_k)$:
\begin{equation} \label{eq:interpr2}
    \h{\eta}_k (\widetilde{T}_k) \lesssim \eta_k(\widetilde{T}_k) + \norm{\u^d}_{0,\widetilde{T}_k}\,,
\end{equation}
where $\eta_k(\widetilde{T}_k)$ has been bounded by \eqref{eq:interpr1}.  Since $\widetilde{T}_k$ will be marked in the $(k+1)$th iteration, it holds that 
$$
\widetilde{T}_k \in \T{k}^0 \quad \text{and} \quad
|\omega_{\widetilde{T}_k}|\lesssim  \norm{h_k}_{\infty,\widetilde{\Omega}_k^0}^3 \to 0 \quad \text{as} \ k \to \infty\,,$$
by \eqref{eq:localuni} and the uniform convergence \eqref{uniformstrong} of adaptive meshes. Taking advantage of the absolute continuity of the Lebesgue integral and Theorem \ref{limiting}, we can get the desired vanishing limit of $\{\h{\eta}_k (\widetilde{T}_k)\}$ from \eqref{eq:interpr1} and \eqref{eq:interpr2}  when $k$ tends to infinity. The proof is complete.
\end{proof}

We are now well-prepared to show that 
the limiting state $\y_\infty^*$ and adjoint state $\p_\infty^*$ actually satisfy the variational problems (\ref{kkt1:1}) and (\ref{kkt1:2}), respectively.
It is worth emphasizing that in our proof, there is no need to introduce a buffer layer of elements between the meshes at different levels as in \cite{xu2017convergent}, by virtue of the generalized convergence result of mesh-size functions \eqref{uniformstrong}. 
\begin{lem}\label{btwo}
Suppose that $(\y_\infty^*,\p_\infty^*,\u_\infty^*)$ is the solution to the limiting optimality system \eqref{kkt3:1}-\eqref{kkt3:3}. Then it satisfies the variational problems \eqref{kkt1:1} and \eqref{kkt1:2}, namely, 
\begin{align}
&B(\y_\infty^*,\varphii)=(\f+\u_\infty^*,\varphii)_{0,\Omega} \qquad \qquad \quad\, \,  \forall \ \varphii \in \textit{\textbf{H}}_0(\curl,\Omega)\,,\label{lemb2:1}\\
&B(\p_\infty^*,\varphii)=(\curl\y_\infty^*-\y^d,\curl\varphii)_{0,\Omega} \quad  \forall \ \varphii \in \textit{\textbf{H}}_0(\curl,\Omega)\,.\label{lemb2:2}
\end{align}
\end{lem}


\begin{proof}
We only prove (\ref{lemb2:1}) since the proof of \eqref{lemb2:2} is similar. For this, 
we first introduce the following residual functionals on $\textit{\textbf{H}}_0(\curl,\Omega)$:  
$$\mathcal{R}(\cdot):=B(\y_\infty^*,\cdot)-(\f+\u_\infty^*,\cdot)_{0,\Omega}\,,$$
and 
$$\mathcal{R}_k(\cdot):=B(\y_k^*,\cdot)-(\f+\u_k^*,\cdot)_{0,\Omega} \,,\ k \ge 0\,.$$
Note that Theorem \ref{limiting} gives us the operator-norm convergence:
\begin{equation} \label{eq:normlim}
    \lim_{k \to \infty}\sup_{\varphii \in \textit{\textbf{H}}_0(\curl,\Omega)}\frac{|\mathcal{R}(\varphii)-\mathcal{R}_k(\varphii)|}{\norm{\varphii}_{\textit{\textbf{H}}_0(\curl,\Omega)}} = 0\,.
\end{equation}
Therefore, to prove $\mathcal{R}$ is actually a zero functional, it is sufficient to show
\begin{equation*}
    \lim_{k\to \infty} \mathcal{R}_k(\varphii) = 0\quad \forall \ \varphii \in \textit{\textbf{C}}_c^\infty(\Omega)\,,
\end{equation*}
by the operator-norm convergence \eqref{eq:normlim} and the density of $\textit{\textbf{C}}_c^\infty(\Omega)$ in $\textit{\textbf{H}}_0(\curl,\Omega)$. Recall the standard interpolation operator $\widetilde{\Pi}_h: \textit{\textbf{H}}^s(\curl,G) \rightarrow \widetilde{\V}_h(G)$ \mb{for $1/2 + \delta \le s \le 1, \delta > 0$, associated with N\'{e}d\'{e}lec's edge elements}, and the corresponding error estimate (cf.\,\cite[Theorem 5.41]{monk2003finite}):
\begin{equation}\label{inpo:sta}
    \norm{\v-\widetilde{\Pi}_h\v}_{\curl,G}\lesssim h^s\left(\norm{\v}_{s,G}+\norm{\curl \v}_{s,G}\right),
    \end{equation}
where $G$ is a polyhedral Lipschitz subdomain of $\Omega$ and 
 $\widetilde{\V}_h(G)$ is the lowest-order conforming edge element space without the specified  boundary condition. Defining $\w:= \varphii - \widetilde{\Pi}_k \varphii \in \textit{\textbf{H}}_0(\curl,\Omega)$ and applying the quasi-interpolation operator $\Pi_k$ introduced in Lemma \ref{interpo} to $\w$, we have, by 
definition of $\mathcal{R}_k$ and $\widetilde{\Pi}_k\varphii,\  \Pi_k \w \in \V_k$, 
\begin{equation*}
\mathcal{R}_k(\varphii) =  B(\yk,\varphii)-(\f+\uk,\varphii)_{0,\Omega} = B(\yk,\w-\Pi_h\w)-(\f+\uk,\w-\Pi_h\w)_{0,\Omega}\,.
\end{equation*}
\mb{By the splitting of the mesh $\T{k}$ introduced in \eqref{eq:splitting}}, a derivation similar to the one for (\ref{reab:1}) and (\ref{reab:2})
gives 
\begin{align} \label{lemb2pr:1}
|\mathcal{R}_k(\varphii)| = |\mathcal{R}_k(\w-\Pi_h\w)| & \lesssim \sum_{T \in \mathscr{T}_k}\left(\eta_{y,T}^{(1)}+\eta_{y,T}^{(2)}\right)\norm{\w}_{\curl,\widetilde{\Om}_T}
+ \sum_{T \in \mathscr{T}_k} \sum_{F \in \partial T \cap \Omega}\left(\eta_{y,F}^{(1)}+\eta_{y,F}^{(2)}\right)\norm{\w}_{\curl,\widetilde{\Om}_T} \notag \\
& \lesssim \sum_{T \in \mathscr{T}_l^+}\eta_k(T)\norm{\varphii - \widetilde{\Pi}_k \varphii}_{0,\widetilde{\Omega}_T}+\sum_{T \in \mathscr{T}_k \backslash \mathscr{T}_l^+}\eta_k(T)\norm{\varphii - \widetilde{\Pi}_k \varphii}_{0,\widetilde{\Omega}_T}\,,
\end{align}
where $l$ is an iteration index less than $k$ (clearly, $\T{l}^+ \subset \T{k}$). \mb{Note that 
\begin{align*}
    \eta_k(T) \le \h{\eta}_k(\widetilde{T}_k) = \max\limits_{T\in \mathscr{T}_k}\h{\eta}_k(T)
\end{align*}
holds for all $T \in \T{l}^+$, where $\widetilde{T}_k$ is as defined in Lemma \ref{lem:vaniestimator}. We also observe from \eqref{inpo:sta} that}
\begin{align*}
    \norm{\varphii - \widetilde{\Pi}_k \varphii}_{0,\widetilde{\Omega}_T} \lesssim \norm{\varphii}_{H^s(\curl,\Om)}
\end{align*}
Hence, by these two observations and Cauchy's inequality, \eqref{lemb2pr:1} implies 
\begin{equation} \label{lemb2pr:2}
    |\mathcal{R}_k(\varphii)| \lesssim  \# (\mathscr{T}_l^+) \h{\eta}_k(\widetilde{T}_k)\norm{\varphii}_{H^s(\curl,\Omega)}+ \Big(\sum_{T \in \mathscr{T}_k \backslash \mathscr{T}_l^+}\eta^2_k(T)\Big)^{1/2}\norm{\varphii - \widetilde{\Pi}_k \varphii}_{0,\widetilde{\Omega}^0_l}\,.
\end{equation}
\mb{To derive the second term in the right-hand side of \eqref{lemb2pr:2}, we have also used another observation, by definition and the shape regularity of the meshes, 
that for any $T \in \mathscr{T}_k \backslash \mathscr{T}_l^+$, there exists an element $T' \in \mathscr{T}_l^0$ such that $T \subset T'$ and the extended neighborhood $\widetilde{\Omega}_T$ of $T$ in $\mathscr{T}_k$ is contained in the extended neighborhood $\widetilde{\Omega}_{T'}$ of $T'$ in $\mathscr{T}_l$, i.e.,
$\widetilde{\Omega}_T \subset \widetilde{\Omega}_{T'}$, which, along with the properties \eqref{eq:localuni2}, allows us to write 
\begin{align*}
    \sum_{T \in \mathscr{T}_k \backslash \mathscr{T}_l^+} \norm{\varphii - \widetilde{\Pi}_k \varphii}^2_{0,\widetilde{\Omega}_T} = \sum_{T' \in \mathscr{T}_l^0} \sum_{T \in \mathscr{T}_k, T \subset T'} \norm{\varphii - \widetilde{\Pi}_k \varphii}^2_{0,\widetilde{\Omega}_T} \le C \sum_{T' \in \mathscr{T}_l^0} \norm{\varphii - \widetilde{\Pi}_k \varphii}^2_{0,\widetilde{\Omega}_{T'}}\,,
\end{align*}
with $C$ independent of the meshes.} For a fixed $l$, by Lemma \ref{lem:vaniestimator},
the first term in \eqref{lemb2pr:2} vanishes when $k$ tends to infinity.
 For the second term in \eqref{lemb2pr:2}, we note from the inequality (\ref{mpost}) and the boundedness of numerical solutions $\{(\yk,\pk,\uk)\} $  that
\begin{equation} \label{aux_eq13}
    \sum_{T \in \mathscr{T}_k \backslash \mathscr{T}_l^+}\eta^2_k(T) \le \sum_{T \in \mathscr{T}_k}\eta^2_k(T) \le C\,,
\end{equation}
where the constant $C$ is independent of $k$. We hence have from \eqref{lemb2pr:2} that
\begin{equation} \label{eq:limrkpr}
    \limsup_{k \rightarrow \infty} |\mathcal{R}_k(\varphii)| \lesssim \limsup_{k \rightarrow \infty} \norm{\varphii - \widetilde{\Pi}_k \varphii}_{0,\widetilde{\Omega}^0_l}\,.
\end{equation}
We note that $\widetilde{\Om}_l^0$ inherits a triangulation from the mesh $\T{k}$ in the sense that $\widetilde{\Om}_l^0 = \bigcup_{T \in \T{k}, T \subset \widetilde{\Om}_l^0} T$. Then the interpolation error estimate \eqref{inpo:sta} gives
\begin{equation*}
    \norm{\varphii - \widetilde{\Pi}_k \varphii}_{0,\widetilde{\Omega}^0_l} \lesssim \norm{h_l}^s_{\infty,\widetilde{\Om}_l^0}\norm{\varphii}_{H^s(\curl,\Om)}\,,
\end{equation*}
which, by \eqref{eq:limrkpr}, implies
\begin{align} \label{aux_eq17}
    \limsup_{k \rightarrow \infty} |\mathcal{R}_k(\varphii)| \lesssim \norm{h_l}^s_{\infty,\widetilde{\Om}_l^0}\norm{\varphii}_{H^s(\curl,\Om)}\,.
\end{align}
Letting $l \to \infty$ in \eqref{aux_eq17} and recalling the uniform convergence of mesh funcitons \eqref{uniformstrong}, we can readily conclude
that for all $\varphii \in \textit{\textbf{C}}_c^\infty(\Omega)$, there holds
\begin{equation*}
\mathcal{R}(\varphii) = \lim_{k\to \infty} \mathcal{R}_k(\varphii) = 0\,,
\end{equation*}
which completes the proof.
\end{proof}

Lemma \ref{btwo} suggests that it suffices to show that $\u_\infty^*$ is actually the minimizer to the cost functional $J(\u)$ over the admissible set $\U^{ad}$. For this, we prove the  following result by exploiting the similar splitting of $\T{k}$ as in Lemma \ref{btwo} and some fundamental properties of $L^2$-projections. 


\begin{thm} \label{thm:mainresult}
Under the same assumptions as in Lemma \ref{btwo},   
$\u_\infty^*$ satisfies the variational inequality:
\begin{equation} \label{aux_eq12}
(\p_\infty^*+\alpha(\u_\infty^*-\u^d),\u-\u_\infty^*)_{0,\Omega}\geq 0  \quad \forall \  \u \in \U^{ad}\,.
\end{equation}
\end{thm}
\begin{proof}
Note that the convex set $\D_+(\Omega):=\{\v \in  \textit{\textbf{C}}_c^\infty(\Omega)\  | \ \v \ge \0 \}$ is dense in $\U^{ad}$ with respect to the $L^2$-norm. It suffices to show that the variational inequality \eqref{aux_eq12} holds for all $\u$ in $\D_+(\Omega)$. However, we have
\begin{equation*}
(\p_\infty^*+\alpha(\u_\infty^*-\u^d),\u-\u_\infty^*)_{0,\Omega}\geq 0  \quad \forall \ \u \in \U_\infty^{ad}\,,
\end{equation*}
since $(\y_\infty^*,\p_\infty^*,\u_\infty^*)$ is the solution to the limiting optimality system (\ref{kkt3:1})-(\ref{kkt3:3}). The above fact, along with the relation $\U_k^{ad} \subset \U_{\infty}^{ad}$, implies that for all $\u \in \D_+(\Omega)$, there holds
\begin{equation}\label{prlo:1}
(\p_\infty^*+\alpha(\u_\infty^*-\u^d),\u-\mathbb{P}_k\u+\mathbb{P}_k\u-\u_\infty^*)_{0,\Omega} \ge (\p_\infty^*+\alpha(\u_\infty^*-\u^d),\u-\mathbb{P}_k\u)_{0,\Omega}.
\end{equation}
It follows from the strong convergence of $\{(\y^*_k,\p^*_k,\u^*_k)\}$ and $\norm{\u-\mathbb{P}_k\u}_{0,\Omega} \le 2\norm{\u}_{0,\Omega}$ that
\begin{equation}\label{prlo:2}
\liminf_{k \rightarrow \infty}  (\p_\infty^*+\alpha(\u_\infty^*-\u^d),\u-\mathbb{P}_k\u)_{0,\Omega} = \liminf_{k \rightarrow \infty} (\pk+\alpha(\uk-\u^d),\u-\mathbb{P}_k\u)_{0,\Omega}\,.
\end{equation}
Noting that $\mathbb{I}-\mathbb{P}_k$ is an orthogonal $L^2$-projection and using the Poincar\'{e} inequality for $(\mathbb{I}-\mathbb{P}_k)\u$:
\begin{equation*}
    \norm{(\mathbb{I}-\mathbb{P}_k)\u}_{0,T} \lesssim h_T \norm{\nabla \u}_{0,T}\,,
\end{equation*}
we derive
\begin{align}
&|(\pk+\alpha (\uk-\u^d),\u-\mathbb{P}_k\u)_{0,\Omega}|=|((\mathbb{I}-\mathbb{P}_k)(\pk-\alpha\u^d),(\mathbb{I}-\mathbb{P}_k)\u)_{0,\Omega}| \notag \\
\lesssim &\sum_{T\in \mathscr{T}_k} (h_T\norm{\pk-\mathbb{P}_k\pk}_{0,T} + h_T\norm{\u^d-\mathbb{P}_k\u^d}_{0,T}) \norm{\nabla \u}_{0,T}\notag\\
\lesssim &\sum_{T\in \mathscr{T}_l^+} h_T \cdot \h{\eta}_k(T) \norm{\nabla \u}_{0,T} +\sum_{T\in \mathscr{T}_k \backslash \mathscr{T}_l^+}h_T \cdot \h{\eta}_k(T) \norm{\nabla \u}_{0,T}\,, \label{eq:estforvi}
\end{align}
by the definition of $\h{\eta}_k(T)$. Here we have used the same splitting of $\T{k}$ as the one in Lemma \ref{btwo} with $l$ being a fixed iteration index less than $k$. Similarly to the proof of Lemma \ref{btwo}, we estimate the two terms in \eqref{eq:estforvi}, respectively, as follows.  For the first term in \eqref{eq:estforvi}, we have
\begin{equation*}
    \sum_{T\in \mathscr{T}_l^+} h_T \cdot \h{\eta}_k(T) \norm{\nabla \u}_{0,T} \lesssim \#(\T{l}^+) \h{\eta}_k(\widetilde{T}_k)\norm{\nabla \u}_{0,\Omega}\,,
\end{equation*}
which, by Lemma \ref{lem:vaniestimator}, vanishes as $k \to \infty$. For the second term in \eqref{eq:estforvi}, note from \eqref{aux_eq13} that 
\begin{align*}
    \sum_{T \in \mathscr{T}_k \backslash \mathscr{T}_l^+}\h{\eta}^2_k(T) = \sum_{T \in \mathscr{T}_k \backslash \mathscr{T}_l^+}\eta^2_k(T) + \osc_T^2(\u^d) \le C
\end{align*}
holds with the constant $C$ independent of $k$. Then, by
 Cauchy's inequality, we have
\begin{equation*} 
    \sum_{T\in \mathscr{T}_k \backslash \mathscr{T}_l^+}h_T \cdot \h{\eta}_k(T) \norm{\nabla \u}_{0,T}\lesssim  \norm{h_l}_{\infty,\Omega_l^0}\norm{\nabla \u}_{0,\Omega}\,,
\end{equation*}
which vanishes when $l$ tends to infinity.
  By above arguments, we have proven 
\begin{equation} \label{prlo:3}
\lim_{k \rightarrow \infty} |(\pk+\alpha(\uk-\u^d),\u-\mathbb{P}_k\u)_{0,\Omega}| = 0\,.
\end{equation}
In view of (\ref{prlo:1}), (\ref{prlo:2}) and \eqref{prlo:3}, we complete the proof, \mb{since the left-hand side of \eqref{prlo:1} is independent of $k$.}
\end{proof}

We end our theoretical analysis for the adaptive approximation of the control problem with an additional remark.
\begin{re}
If the function $\u^d$ has the regularity: $\u^d \in \textit{\textbf{H}}^1(\Omega)$, then using Theorem \ref{thm:converge}, the uniform convergence \eqref{uniformstrong} and the Poincar\'{e} inequality, we can argue in a manner similar to \cite[Theorem 6.2]{xu2017convergent}
to conclude the convergence of the error estimators $\{\h{\eta}_k\}$. However, for a general $L^2$-data $\u^d$, the optimal control $\u^*$ can only be guaranteed to have a $L^2$-regularity, and the data oscillation $\{\osc_k(\u^d)\}$ may not converge to zero in the adaptive process. 
\end{re}

\section{Numerical experiments} \label{experiment}
In this section, we provide a detailed documentation of the numerical results to illustrate the performance of our adaptive algorithm for the optimal control problem. All the examples presented below are implemented based on the MATLAB package iFEM \cite{Chen.L2008c} using MATLAB 2017a on a personal laptop with 8.00 GB RAM and dual-core 2.4 GHz CPU.  When solving the discrete optimization problems \eqref{discrete problem1}-\eqref{discrete problem2},  we use the projected gradient algorithm in which the HX-preconditioner \cite{hiptmair2007nodal} is also adopted for solving the $\textit{\textbf{H}}(\curl)$-elliptic problem. More precisely, 
we start with an initial guess $\u_h^{\sss (0)} \in \U_h^{ad}$, and in $n$th step ($n \ge 1$) we solve two $\textit{\textbf{H}}(\curl)$-elliptic problems to obtain the state $\y_h^{\sss (n)}$ and the adjoint state $\p_h^{\sss (n)}$ with the help of HX-preconditioner. Then the control variable $\u_h$ can be explicitly updated as follows: 
\begin{equation} \label{eq:projstep}
    \u_h^{(n)} = \max\left\{\0, \u_h^{(n-1)} + s\left(-\mathbb{P}_h\p_h^{(n)} - \alpha \left(\u_h^{(n-1)} - \u_h^d\right)\right)\right\}\,,
\end{equation}
where the parameter $s$ can be computed explicitly by solving a one dimensional quadratic optimization problem. Compared to the algorithm in \cite{hoppe2015adaptive} where the edge element was used to approximate the control and an additional least squares problems with inequality constraints need to be solved to realize the box constraint \eqref{eq:addset}, which is very expensive and time-consuming, our algorithm uses only about 
half the number of DoFs for computing the control variable but gives equally accurate numerical resolution, and can be trivially implemented. 


In the following, we carry out some numerical experiments for two benchmark problems. For both examples, we shall first compute the numerical solutions and errors on the uniform meshes with five refinement iterations for the purpose of comparison, and then conduct the new adaptive algorithm, which is terminated when the corresponding DoFs are almost the same as that of the finest uniform mesh. The first example is modified from \cite{zhong_convergence_2010}\cite{duan2016adaptive}, where there is a corner singularity in the solution. 
We consider an optimal control problem on the L-shape domain: $\Omega = [-1,1]^3\backslash[0,1]\times[0,1]\times[-1,1]$  with both the electric permittivity $\sigma$ and the magnetic permeability $\mu$ taken to be $1$. We set the target applied current density (control) $\u^d$ and the target magnetic field (state) $\y^d$ to zero, and set the parameter $\alpha$ to $0.1$. We choose the inhomogeneous Dirichlet boundary condition and the source term $\f$ such that the exact solution are given as follows:
\begin{equation*}
\mb{\u^* = \p^* = 0 \quad \text{and}\  \y^* = {\bf grad}(r^{\frac{2}{3}}\sin(\frac{2}{3}\theta))\ \text{in cylindrical  coordinates}\,.}
\end{equation*}
\begin{figure}[htbp]
\includegraphics[clip,width=0.3\textwidth]{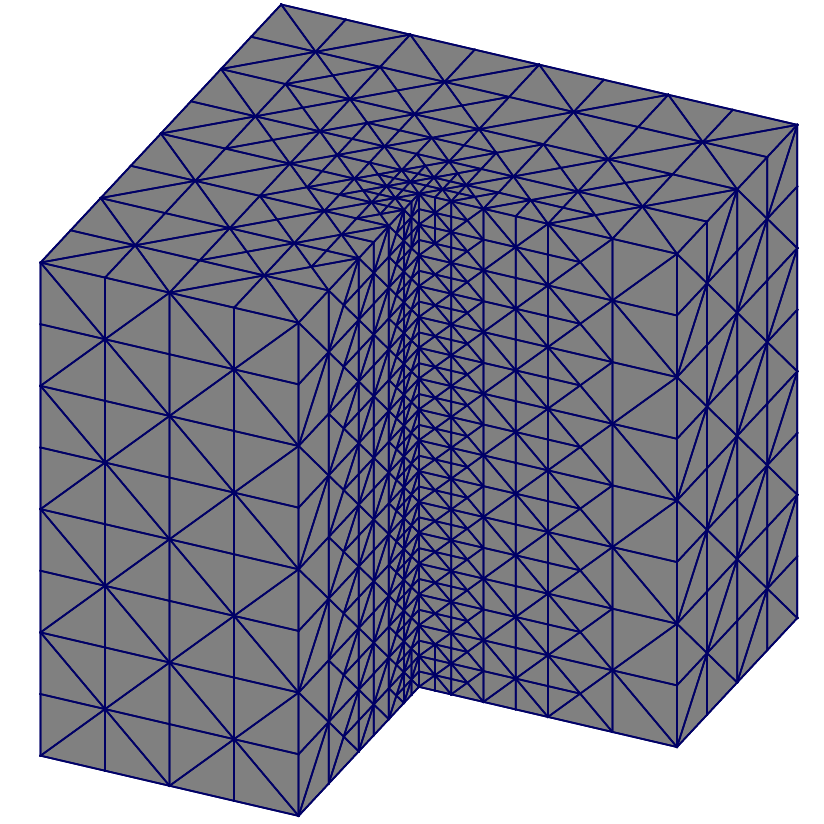}\hfill{}
\hfill{}\includegraphics[clip,width=0.3\textwidth]{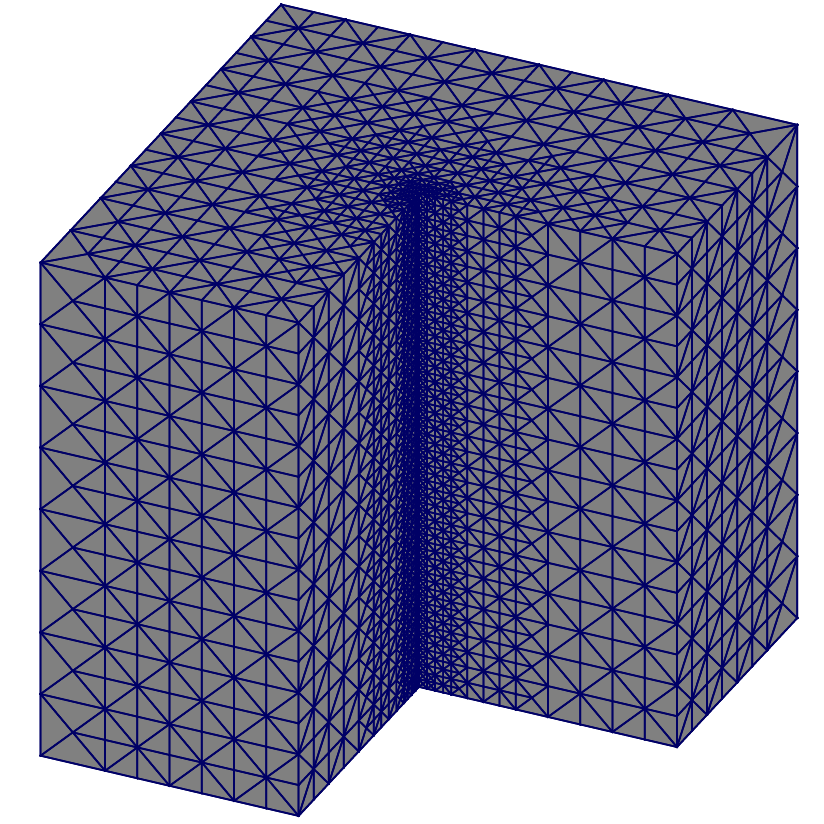}\hfill{}
\hfill{}\includegraphics[clip,width=0.3\textwidth]{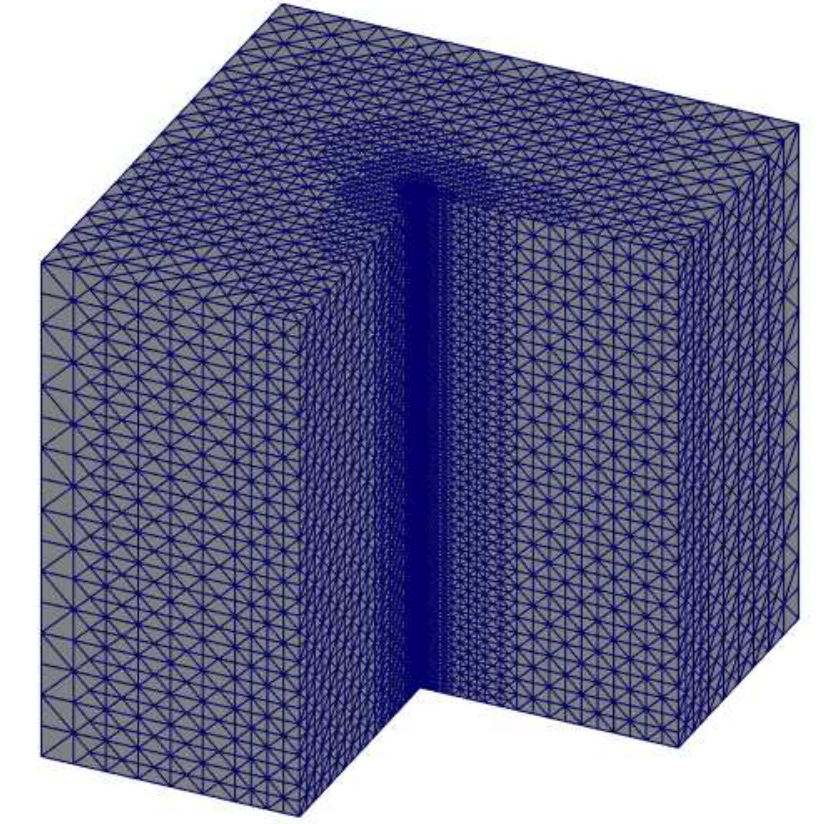}\hfill{}
\caption{\label{fig:adapmeshex1}
Evolution of the adaptive mesh for the first example in steps $k = 10,14$ and $18$ (final mesh).
}
\end{figure}
We set $\theta = 0.5$ in the marking strategy \eqref{mark1} and remark that there is no need to consider the data oscillation $\osc_T(\u^d)$ in this example because of the vanishing desired control $\u^d$.

In our simulations, the computation starts on a very coarse mesh with $480$ DoFs.  The adaptively refined meshes in the 
$10$th, $14$th and $18$th iterations are presented in Figure \ref{fig:adapmeshex1}, in which we clearly observe that the meshes are locally refined around the $z$-axis where the singularity of $\y^*$ occurs.
Note the non-smooth function $\y^* \notin \textit{\textbf{H}}^1(\Omega)$ and that the lowest-order edge elements of first family are used to discretize the state variable. We typically cannot 
expect the optimal convergence order $O(h)$ on a  uniformly refined mesh according to the standard interpolation result (\ref{inpo:sta}). However, this theoretical order of convergence can be recovered by the adaptive mesh refinement in the sense that
\begin{equation*}
    \norm{\yk-\y^*}_{\curl,\Omega} \lesssim (\#{\rm DoFs})^{-1/3}\,,
\end{equation*}
which has been confirmed by the convergence history plotted in Figure \ref{fig:conex1} (left) on a double logarithmic scale. We should also observe that the errors of the control $\norm{\u^*-\uk}_{0,\Omega}$ and the adjoint state
$\norm{\p^*-\pk}_{\curl,\Omega}$ reduce with a slope $-2/3$, which is twice the one of the error associated with the state $\y$.  The difference in the convergence rates may be explained by the fact that $\u^*$ and $\p^*$ are zero functions and hence very smooth while the optimal state $\y^*$ is non-smooth and has a strong singularity near the $z$-axis. Moreover, Figure \ref{fig:conex1} (right) shows the reduction of the total error (almost dominated by the error $\norm{\y^*-\yk}_{\curl,\Om}$):
 \begin{equation*}
 \norm{\y^*-\yk}_{\curl,\Omega}+\norm{\p^*-\pk}_{\curl,\Omega}+\norm{\u^*-\uk}_{0,\Omega}\,,
 \end{equation*}
and the convergence behavior of the a posteriori error estimator $\h{\eta}_h = \eta_h$, which confirms the efficiency of $\h{\eta}_h$ and the superiority of the adaptive mesh refinement over the uniform mesh refinement: the total error on the adaptively refined mesh reduces with an order $-1/3$, double the one ($-0.15$) on the uniformly refined mesh. This is also confirmed by the computing times and computational costs: 
for achieving the error over the mesh generated by the $5$th uniform refinement, 
the adaptive algorithm takes only about $80$ seconds and $10^5$ DoFs, 
whereas the uniform one takes about $580$ seconds and $1410240$ DoFs.

\begin{figure}[htbp]
\hfill{}\\
\hfill{}\includegraphics[clip,width=0.5\textwidth]{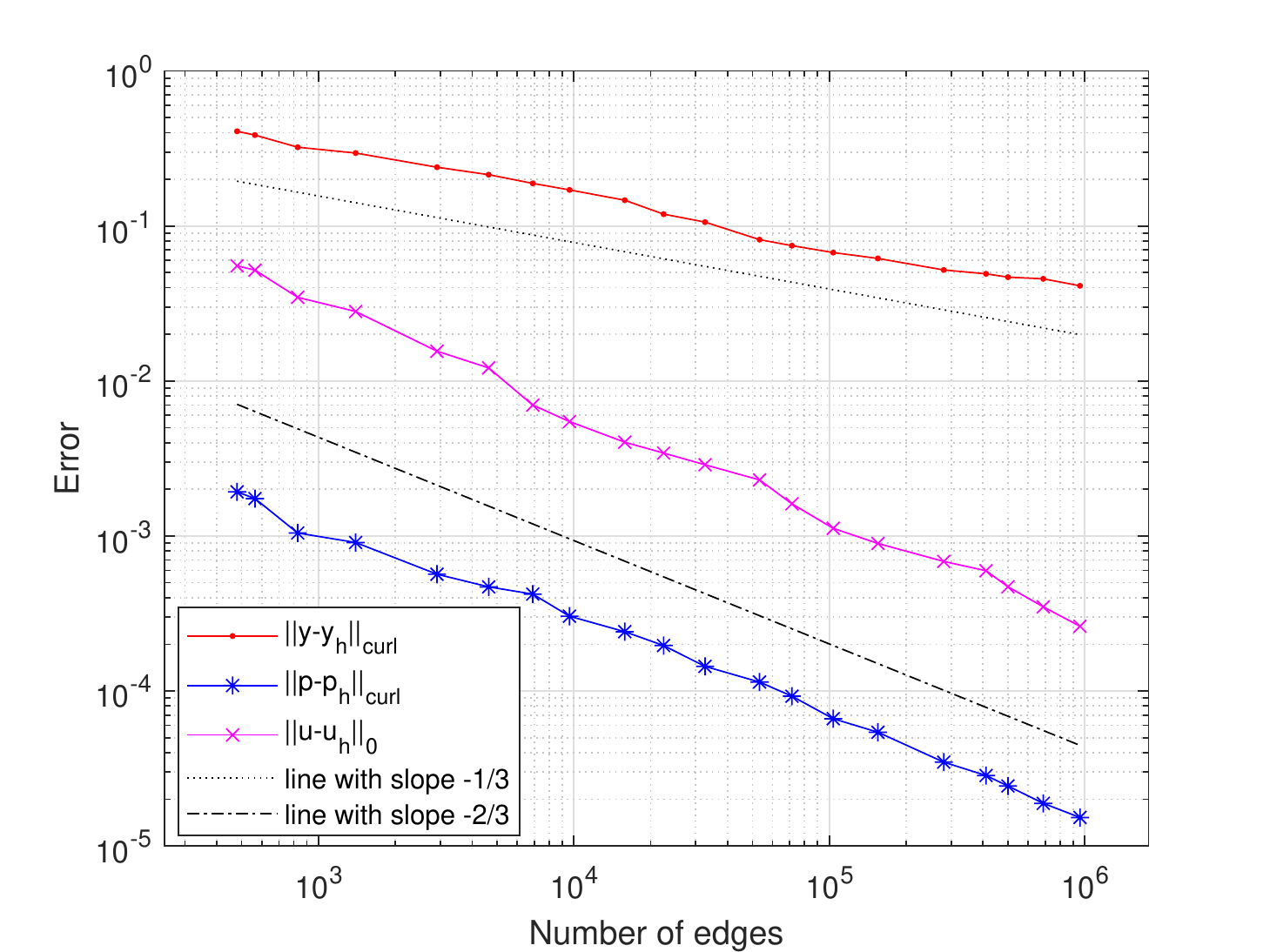}\hfill{}
\hfill{}\includegraphics[clip,width=0.5\textwidth]{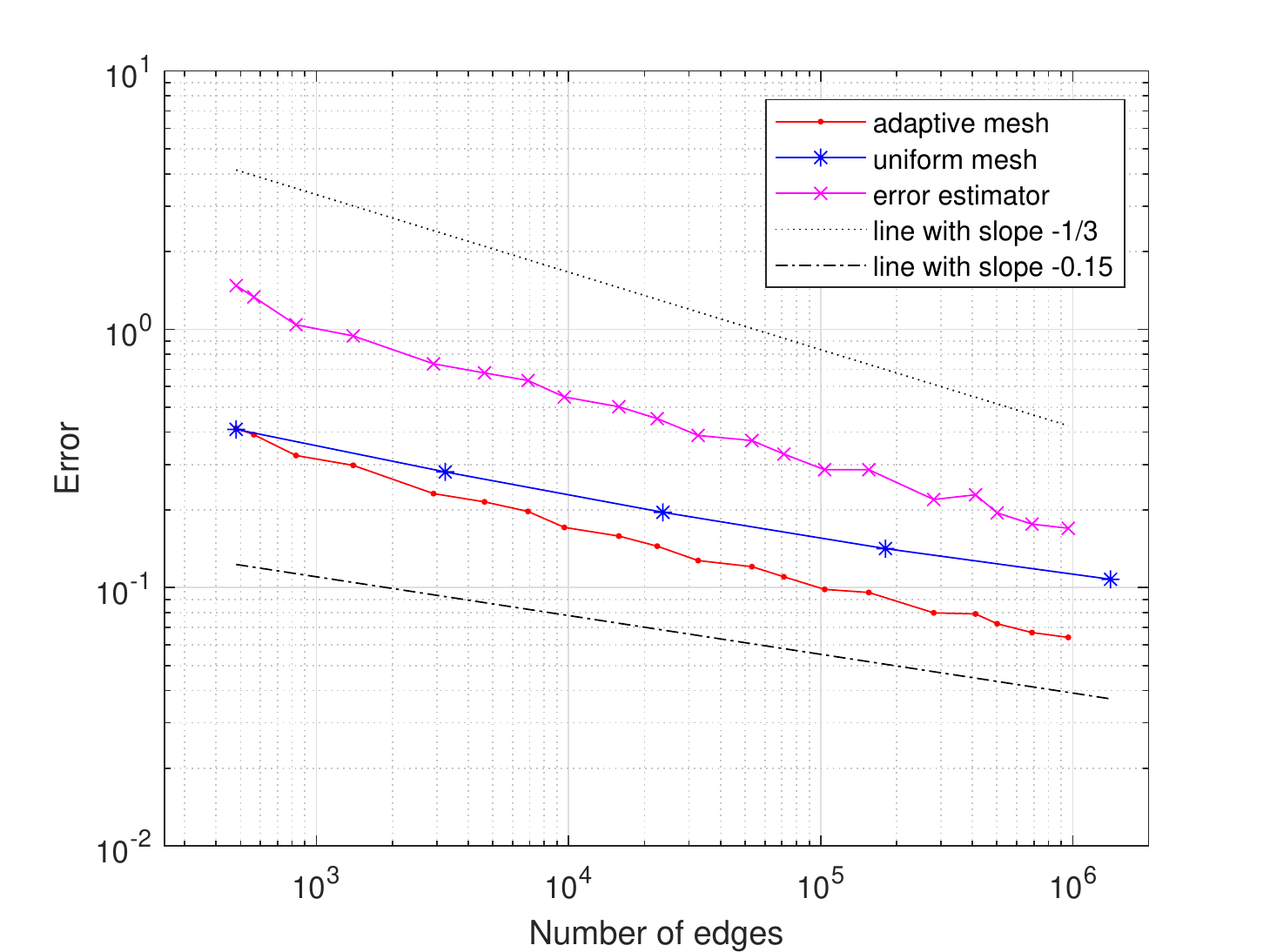}\hfill{}
\caption{\label{fig:conex1} Convergence histories of the control, state and adjoint state (left), and the total errors for the uniform mesh refinement and the adaptive mesh refinement, as well as the error estimator $\h{\eta}_h$ (right) for the first example.}
\end{figure}

The second example is chosen to be similar to the one in \cite{hoppe2015adaptive}, where there are non-smooth source terms and large jumps of physical coefficients across the interface between two different media. To be precise, we consider an optimal control problem on $\Omega = [-1,1]^3$ with a high-contrast inclusion: $\Omega_c := \{\x\in \R^3\,;\ x^2 + y^2 + z^2 < 0.6^2\}$, on which the coefficients $\mu$ and $\sigma$ are given by 
\begin{equation*}
  \sigma =
   \begin{cases}
   10 &\mbox{in\  $\Omega_c$}\,,\\
   1 &\mbox{in\  $\Omega \backslash \Omega_c$}\,,
   \end{cases} \quad \mu^{-1} =
   \begin{cases}
   0.1 &\mbox{in\  $\Omega_c$}\,,\\
   1 &\mbox{in\  $\Omega \backslash \Omega_c$}\,.
   \end{cases}
\end{equation*}
We set the target state $\y^d = \0$ and the 
target control $\u^d  = 10(\chi_{\Omega_c},0,0)$. Here $\chi_{\Omega_c}$ is the characteristic function of $\Omega_c$. We define a scalar function:
\begin{equation*}
\mb{\phi(\x) = \frac{1}{2\pi}\sin(2\pi x)\sin(2\pi y)\sin(2\pi z)\,,}
\end{equation*}
and then introduce the non-smooth  source term $\f$: 
\begin{equation*}
    \f(\x) = \sigma \nabla \phi(\x) - 10 (\chi_{\Omega_c},0,0)\,.
\end{equation*}
We readily see that the unique solution to the optimality system \eqref{kkt1:1}-\eqref{kkt1:3} is given by 
\begin{align*}
    \y^* = \nabla \phi\,, \ \p^* = 0\,, \ \text{and}\ \u^* = 10(\chi_{\Omega_c},0,0)\,.
\end{align*}
Other parameters are chosen to be same as the first example. 

\begin{figure}[htbp]
\includegraphics[clip,width=0.23\textwidth]{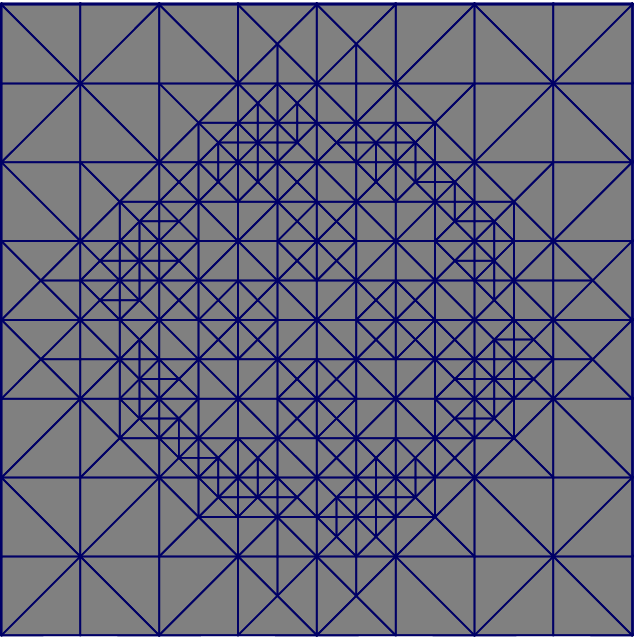}\hfill{}
\includegraphics[clip,width=0.23\textwidth]{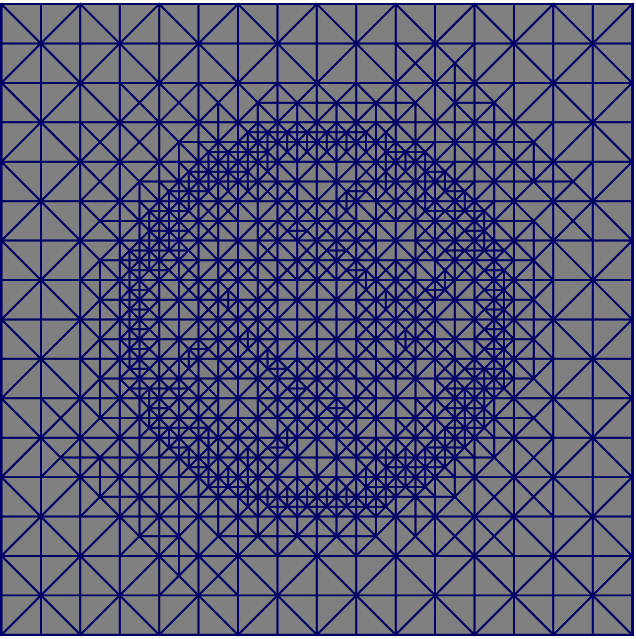}\hfill{}
\includegraphics[clip,width=0.23\textwidth]{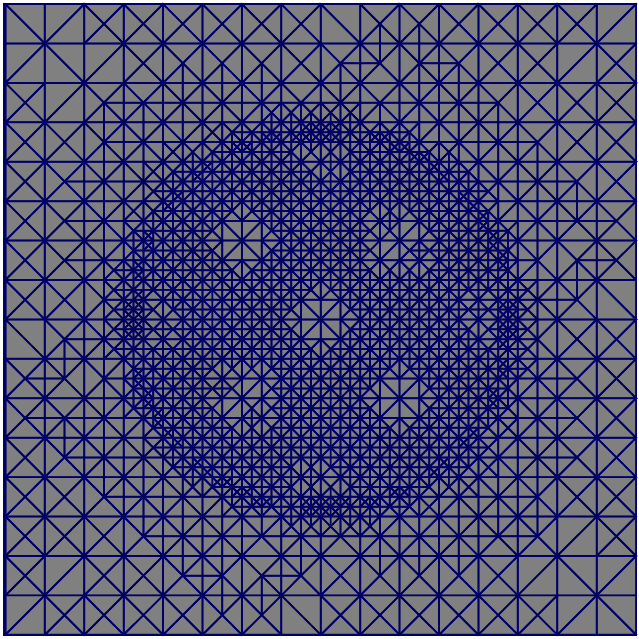}\hfill{}
\includegraphics[clip,width=0.23\textwidth]{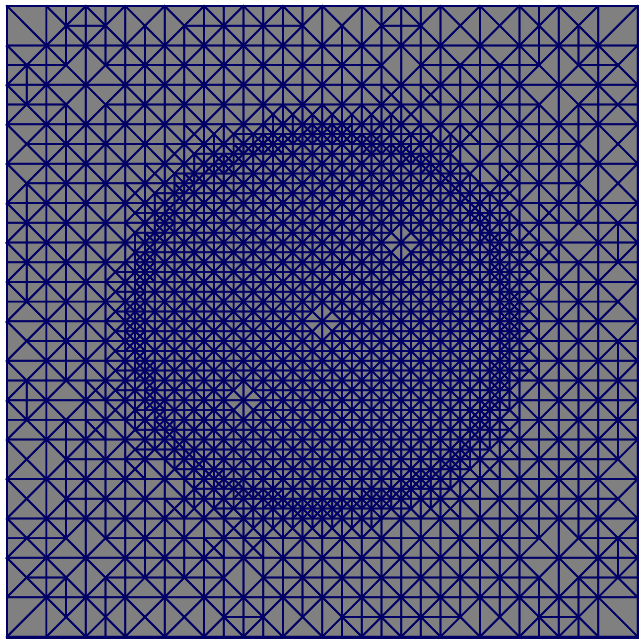}
\caption{\label{fig:adapmeshex2}
Evolution of the adaptively refined mesh (2D slice) for the second example in steps $k = 8$, $12$, $16$ and $19$ (final mesh).}
\end{figure}

We still set $\theta = 0.5$ in the marking strategy \eqref{mark1} and start our computation on a very coarse initial mesh with 604 DoFs. We show the evolution of 
adaptive meshes on the $(y,z)$-cross section at $x = 0$ in Figure \ref{fig:adapmeshex2}, from which we immediately see that the meshes are strongly concentrated in the high-contrast inclusion $\Omega_c$
and clearly capture the shape of the interface. We can also observe, from Figure \ref{fig:converge2} (left), an optimal convergence order $-1/3$ for the control and state variable and a little bit faster decrease of the error for the adjoint state $\norm{\p^*-\pk}_{\curl,\Omega}$, which is  possibly because $\p^*$ is quite smooth (a zero function), compared to $\u^*$ and $\y^*$. Figure \ref{fig:converge2} (right) again verifies the effectiveness of the error estimator $\h{\eta}_k$ and the gain of computational efficiency from the adaptive mesh refinement: the total error on the adaptively refined mesh reduces with an order $-1/3$, whereas the error on the uniform one reduces only with an order $-0.2$. 



\begin{figure}[htbp]
\hfill{}\\
\hfill{}\includegraphics[clip,width=0.5\textwidth]{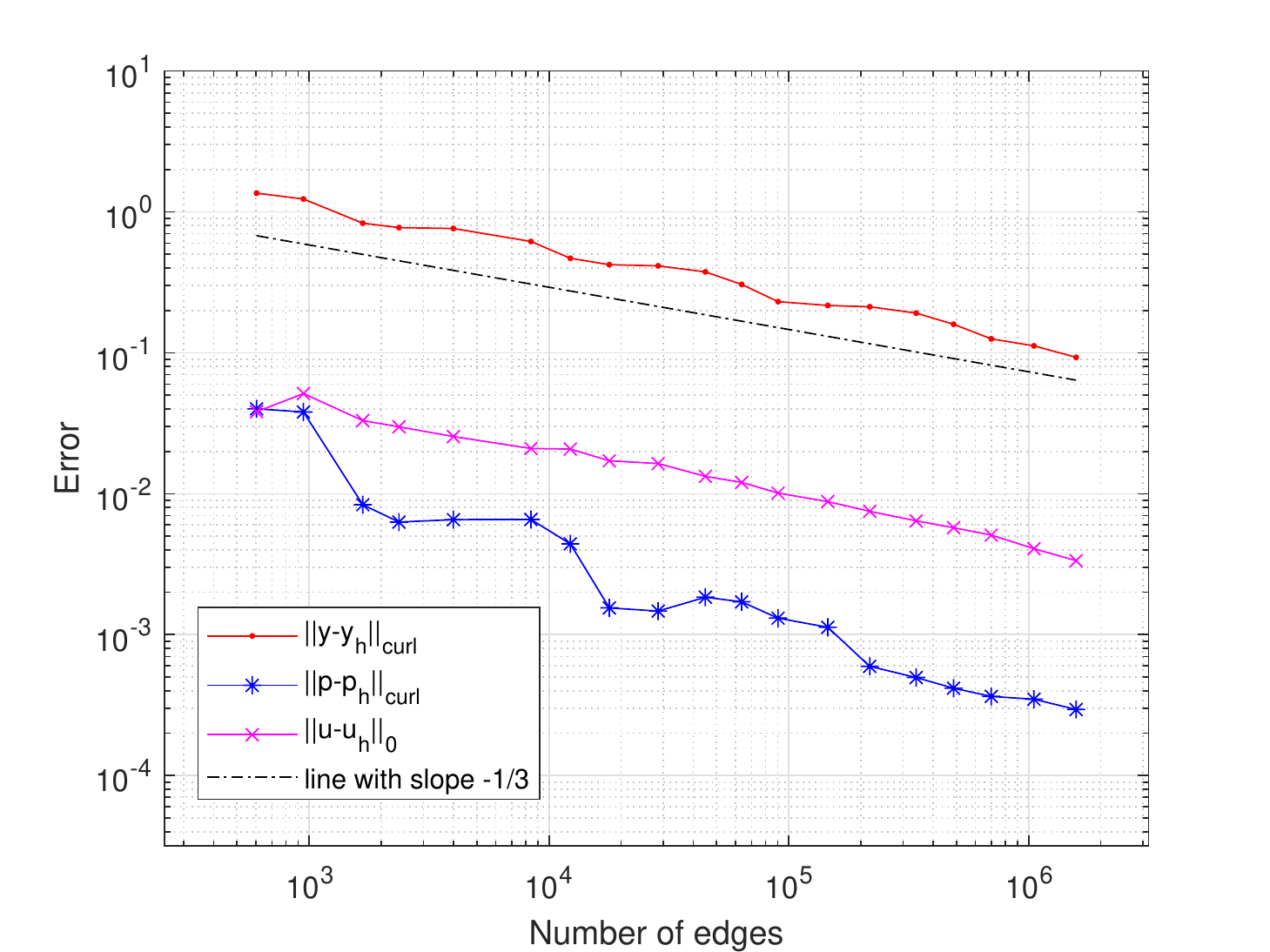}\hfill{}
\hfill{}\includegraphics[clip,width=0.5\textwidth]{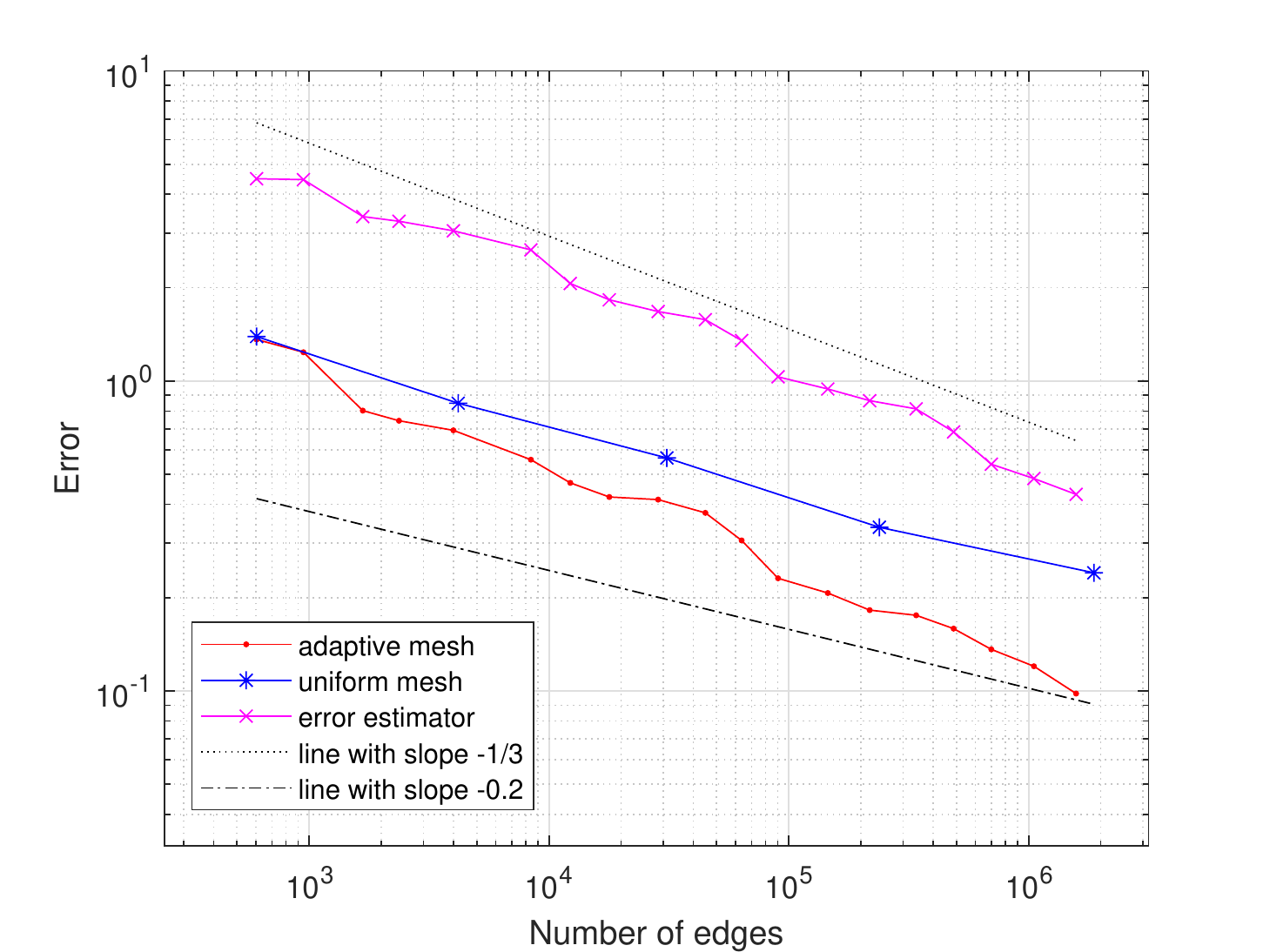}\hfill{}
\caption{\label{fig:converge2} 
Convergence histories of the control, state and adjoint state (left), and the total errors for the uniform mesh refinement and the adaptive mesh refinement, as well as the error estimator $\h{\eta}_h$ (right) for the second  example.
}
\end{figure}

\section{Concluding remarks}
In this work, we have studied an electromagnetic optimal control problem and used the lowest-order edge elements to approximate the state and adjoint state, while used  the piecewise constant functions to approximate the control. We have designed an adaptive finite element method with an error indicator involving both the residual-type error estimator and the lower-order
data oscillation.
We have established the reliability and efficiency of the a posteriori error estimator and the strong convergence of the adaptive finite element solutions for both the state and control variables.
With a very minor modification, our analysis and arguments can be directly applied to the more realistic case \cite{yousept2013optimal}\cite{xu2017convergent} where the control $\u$ is added on a subdomain $\Om_c$ of $\Omega$ and the case \cite{gong2017adaptive} where the control satisfies a bilateral constraint. It is worth pointing out that our arguments can also be modified to cope with the inhomogeneous Dirichlet boundary condition. To be exact,
 suppose that the boundary condition is given by $\gamma_t \y = \gamma_t \g$ with $\g \in \textit{\textbf{H}}(\curl,\Omega)$ being a known function satisfying the given Dirichlet trace data \cite{brenner2007mathematical}.
 Then we can check that the optimality systems \eqref{kkt1:1}-\eqref{kkt1:3} and \eqref{kkt2:1}-\eqref{kkt2:3} still hold, except that \eqref{kkt1:1} and \eqref{kkt2:1} are solved on the affine spaces $\g + \V$ and $\widetilde{\Pi}_h \g + \V_h$, respectively. Hence, up to a possible data oscillation: $\norm{\g - \widetilde{\Pi}_h \g }_{\curl,\Omega}$, our a posteriori error estimates and the convergence analysis can be readily applied. 
 
\mb{As mentioned in the Introduction, the model of our interest can be connected with the discretization of the control problem of time-dependent eddy current equations \cite{kolmbauer2012robust}\cite{nicaise2014two}\cite{nicaise2015optimal}, where the implicit time-stepping scheme is adopted for the sake of stability \cite{hiptmair1998multigrid}\cite{beck2000residual}. In this case, the coefficient $\sigma$ will be scaled by the current time-step size. Therefore, it is important to design an error estimator which is robust with respect to the scaling of the coefficients, namely, the generic constants involved in the a posteriori error analysis should be independent of the scaling factors of the coefficients. For this, we may measure the errors of the states ($\y$ and $\p$) by the energy norm on $\textit{\textbf{H}}_0(\curl,\Omega)$: $\norm{\cdot}^2_{B} = \norm{\sqrt{\mu^{-1}}\curl \cdot}^2_{0,\Omega} + \norm{\sqrt{\sigma}\cdot}_{0,\Omega}^2$, and the error estimators may also need to be scaled correspondingly. If we assume that $\mu$ and $\sigma$ are element-wise constant with respect to the initial mesh $\T{0}$ \cite{beck2000residual}, or that there exists a scaled norm $\|\cdot\|^2_{\V} := \mu_*^{-1}\|{\bf curl} \cdot\|^2_{0,\Omega} + \sigma_* \|\cdot\|_{0,\Omega}^2$ with
$\mu_*$ and $\sigma_*$ being positive constants,  such that $B(\cdot,\cdot)$ is continuous and inf-sup stable with respect to $\|\cdot\|_{\V}$ with the involved generic constants independent of $\mu$ and $\sigma$ \cite{schoberl2008posteriori},  one can naturally follow the a posteriori error analysis provided in this work to obtain a robust error analysis by using the energy norm and the scaled error estimators (for instance, if 
the norm $\norm{\cdot}_\V$ defined above exists, by adding a scaling factor $\sqrt{\mu_*}$, we can modify $\eta_{y,T}^{\sss (1)}$ and $\eta_{y,F}^{\sss (1)}$ as follows: $\eta_{y,T}^{\sss (1)} := h_T \sqrt{\mu_*} \| \f + \u_h^* - \curl \mu^{-1} \curl \y_h^* - \sigma \y_h^*\|_{0,T}$ and $\eta_{y,F}^{\sss (1)} := \sqrt{h_F \mu_*} \|[\gamma_t (\mu^{-1} \curl \y_h^*)]_F\|_{0,F}$). We refer the readers to \cite{cascon2008quasi}\cite{nochetto2009theory} for related discussions. However, the detailed and rigorous treatments for the general coefficients and the time-dependent model are not trivial tasks and need further investigations.} We finally remark that it is also of great interest to design the adaptive algorithm for more complicated models, such as the nonlinear electromagnetic control problem \cite{yousept2017optimal} and the Maxwell variational inequalities \cite{yousept2020well} \cite{yousept2020hyperbolic}. 

\hspace*{2 cm}

\mb{{\bf Acknowledgements:} The work of Jun Zou is substantially supported by Hong Kong Research Grants Council (projects 14306718 and 14306719). The authors 
would like to thank the anonymous referees for their very careful reading of the manuscript and their numerous constructive comments and suggestions, which have helped us improve the presentation and results of this work essentially.}




\end{document}